\documentclass[a4paper]{article}

\usepackage{amsmath,amssymb}
\usepackage{amscd}
\usepackage{mathrsfs}
\usepackage{amsthm}
\usepackage[all]{xy}
\usepackage{color}

\theoremstyle{plain}

\newtheorem{theorem}{\bf Theorem}[section]
\newtheorem{lemma}[theorem]{\bf Lemma}
\newtheorem{proposition}[theorem]{\bf Proposition}

\theoremstyle{definition}

\newtheorem{definition}[theorem]{\bf Definition}

\newtheorem{remark}[theorem]{\bf Remark}

\newcommand{\eqa}[1]{
\begin{align*}
#1
\end{align*}}
\usepackage[all]{xy}
\newcommand{\nai}[2]{\langle #1,#2\rangle}

\newcommand{\mattwo}[4]{\begin{pmatrix}#1 & #2\\ #3 & #4\end{pmatrix}}

\title{On the uniqueness of injective III$_1$ factor}
\date{May 1985}
\author{\large Uffe Haagerup}

\begin{document}
\maketitle

\begin{abstract}
We give a new proof of a theorem due to Alain Connes, that an injective factor $N$ of type III$_1$ with separable predual and with trivial bicentralizer is isomorphic to the Araki--Woods type III$_1$ factor $R_{\infty}$. 
This, combined with the author's solution to the bicentralizer problem for injective III$_1$ factors provides a new proof of the theorem that up to $*$-isomorphism, there exists a unique injective factor of type III$_1$ on a separable Hilbert space. 
\end{abstract}

\noindent

\medskip
\begin{center}
\large Preamble by Alain Connes \normalsize
\end{center}
\begin{quotation}
Uffe Haagerup solved the hardest problem of the classification of factors, 
 namely the uniqueness problem for injective factors of type III$_1$.
 The present paper, taken from his unpublished notes, presents a direct proof 
 of this uniqueness by showing that any injective factor of type III$_1$ is an infinite
 tensor product of type I factors so that the uniqueness follows from the Araki--Woods
 classification. The proof is typical of Uffe's genius, the attack is direct, and combines 
 his amazing control of completely positive maps and his sheer analytical power,
 together with his  solution to the bicentralizer problem.
 After his tragic death, Hiroshi Ando volunteered to type the manuscript\footnote{
The manuscript is typed by Hiroshi Ando (Chiba University) in cooperation with Cyril Houdayer (Universit\'e
Paris-Sud), Toshihiko Masuda (Kyushu University), Reiji Tomatsu (Hokkaido University), Yoshimichi Ueda
(Kyushu University) and Wojciech Szymanski (University of Southern Denmark).  
}.
 Some pages were missing from the notes, but eventually Cyril Houdayer 
 and Reiji Tomatsu suggested a missing proof of Lemma \ref{lem: UH3.4} and Theorem \ref{thm: UH3.1}.
 We heartily thank  Hiroshi, Cyril and Reiji for making the manuscript available to the community.
  We also thank S\o ren Haagerup for giving permission to publish his father's paper.
\end{quotation}
\tableofcontents
\section{Introduction}
The problem, whether all injective factors of type III$_1$ on a separable Hilbert space are isomorphic, has been settled affirmatively. 
The proof of the uniqueness of injective III$_1$ factors falls in two parts, namely (see $\S$\ref{subsec: bicentralizers} for the definition of the bicentralizer): 
\begin{theorem}[{\cite{Connes85}}]\label{thm: A}
Let $M$ be an injective factor of type {\rm{III}}$_1$ on a separable Hilbert space, such that the bicentralizer $B_{\varphi}$ is trivial (i.e., $B_{\varphi}=\mathbb{C}1$) for some normal faithful state $\varphi$ on $M$, then $M$ is $*$-isomorphic to the Araki--Woods factor $R_{\infty}$. 
\end{theorem}
\begin{theorem}[{\cite{Haagerup87}}]\label{thm: B}
For any normal faithful state $\varphi$ on an injective factor $M$ of type {\rm{III}}$_1$ on a separable Hilbert space, one has $B_{\varphi}=\mathbb{C}1$. 
\end{theorem}
In this paper we give an alternative proof of Theorem \ref{thm: A} above, based on the technique of our simplified proof \cite{Haagerup85} of Connes' Theorem \cite{Connes76} ``injective$\Rightarrow$hyperfinite" in the type II$_1$ case\footnote{Typewriter's note: Haagerup used this technique to give a new proof of the uniqueness of injective type III$_{\lambda}\ (0<\lambda<1)$ factor. This result has been published in \cite{Haagerup89}.}.  
The key steps in our proof of Theorem \ref{thm: A} are listed below:\\ \\
\underline{\textbf{Step 1}}\\
By use of continuous crossed products, we prove that the identity map on an injective factor $N$ of type III$_1$ has an approximate factorization
\[
\xymatrix{
&R\ar[rd]^{T_{\lambda}}&\\
N\ar[rr]_{\text{id}_N}\ar[ru]^{S_{\lambda}}&&N
}
\]
through the hyperfinite factor $R$ of type II$_1$, such that $(S_{\lambda})_{\lambda\in \Lambda}$ and $(T_{\lambda})_{\lambda\in \Lambda}$ are nets of normal unital completely positive maps, and for a fixed normal faithful state $\varphi$ on $N$ (chosen prior to $S_{\lambda}$ and $T_{\lambda}$), there exist normal fatihful states $(\psi_{\lambda})_{\lambda\in \Lambda}$ on $R$, such that for all $t\in \mathbb{R}$ and $\lambda \in \Lambda$,  
\begin{align*}
\varphi\circ T_{\lambda}=\psi_{\lambda}&,\ \  \psi_{\lambda}\circ S_{\lambda}=\varphi,\\
\sigma_t^{\varphi}\circ T_{\lambda}&=T_{\lambda}\circ \sigma_t^{\psi_{\lambda}},\\
\sigma_t^{\psi_{\lambda}}\circ S_{\lambda}&=S_{\lambda}\circ \sigma_t^{\varphi},
\end{align*}
and $\|S_{\lambda}\circ T_{\lambda}(x)-x\|_{\varphi}\stackrel{\lambda \to \infty}{\to}0$ for all $x\in N$, where $\|y\|_{\varphi}:=\varphi(y^*y)^{\frac{1}{2}}\ (y\in N)$.\\ \\ 
\underline{\textbf{Step 2}}\\ From Step 1, we deduce that a certain normal faithful state $\varphi$ ($\mathbb{Q}$-stable state defined in $\S$\ref{section: Q-stable state}) on an injective factor $N$ of type III$_1$ has the following property: for any finite set of unitaries $u_1,\dots,u_n$ in $N$ and for every $\gamma,\delta>0$, there exists a finite-dimensional subfactor $F$ of $N$ such that 
\[\varphi=\varphi|_F\otimes \varphi|_{F^\text{c}},\]
and such that there exist unitaries $v_1,\dots,v_n$ in $F$ and a unital completely positive map $T\colon F\to N$ such that 
\begin{align*}
\varphi\circ T=\varphi,\\
\|\sigma_t^{\varphi}\circ T-T\circ \sigma_t^{\varphi|_F}\|&\le \gamma |t|,\ \ \ \ t\in \mathbb{R},
\end{align*} 
and 
\[\|T(v_k)-v_k\|_{\varphi}<\delta,\ \ \ k=1,\dots,n.\]
\ \\  
\underline{\textbf{Step 3}}\\
We prove that if $N,\varphi,F,u_1,\dots,u_n,v_1,\ldots,v_n$ are as in Step 2, then for every $\sigma$-strong neighborhood $\mathcal{V}$ of $0$ in $N$, there exists a finite set of operators $a_1,\dots, a_p$ in $N$ such that 
\begin{list}{}{}
\item[(a)] $\displaystyle \sum_{i=1}^pa_i^*a_i\in 1+\mathcal{V}$\ \ \ and\ \ \ $\displaystyle \sum_{i=1}^pa_i^*a_i\le 1$,
\item[(b)] $\displaystyle \varepsilon_{F,\varphi}\left (\sum_{i=1}^pa_ia_i^*\right )\in 1+\mathcal{V}$\ \ \ and\ \ \ $\displaystyle \varepsilon_{F,\varphi}\left (\sum_{i=1}^pa_ia_i^*\right )\le 1$,
\item[(c)] $\displaystyle \sum_{i=1}^p\|a_i\xi_{\varphi}-\xi_{\varphi}a_i\|^2<\delta',$
\item[(d)] $\displaystyle \sum_{i=1}^p\|a_iu_k-v_ka_i\|_{\varphi}^2<\delta',\ \ \ \ k=1,\ldots,n$. 
\end{list} 
Here $\xi_\varphi$ denotes the unique representing vector of $\varphi$ in a natural cone. 
The above $\delta'>0$ depends on $\gamma$ and $\delta$ in Step 2, and $\delta'$ is small when $\gamma$ and $\delta$ are small. Here, $\varepsilon_{F,\varphi}$ is the $\varphi$-invariant conditional expectation of $N$ onto $F$. Moreover in (c), the standard Hilbert space $H$ of $N$ is regarded as a Hilbert $N$-bimodule, by putting 
$\eta a:=Ja^*J\eta\ (a\in N,\ \eta\in H)$.\\

Assume now that the bicentralizer of any normal faithful state on $N$ is trivial. Then by an averaging argument, we can exchange (b) by
\[\text{(b')}\ \ \ \ \ \sum_{i=1}^pa_ia_i^*\in 1+\mathcal{V}\text{\ \ \ and\ \ \ }\sum_{i=1}^pa_ia_i^*\le 1.\] 
\ \\ \\
\underline{\textbf{Step 4}}\\
From (a), (b'), (c) and (d) above, we derive that there exists a unitary operator $w\in N$ such that 
$$\|w\xi_{\varphi}-\xi_{\varphi}w\|<\varepsilon$$
and
$$\|wu_k-v_kw\|_{\varphi}<\varepsilon,\ \ \ \ \ k=1,\ldots,n,$$
where $\varepsilon$ is small when $\delta'$ is small and $\mathcal{V}$ is a small $\sigma$-strong neighborhood of $0$ in $N$. The key part of Step 4 is a theorem about general Hilbert $N$-bimodules, which was proved in \cite{Haagerup89}.   
\\ \\
\underline{\textbf{Step 5}}\\
From Step 4, we get that for every finite set of unitaries $u_1,\ldots,u_n\in N$ and every $\varepsilon>0$, there exists a finite dimensional subfactor $F_1$ (namely $w^*Fw$) of $N$ and $n$ unitaries $v_1',\ldots,v_n'$ in $F_1$ (namely $w^*v_kw,\ k=1,\ldots,n$), such that 
$$\text{(i)}\ \ \ \ \|v_k'-v_k\|_{\varphi}<\varepsilon$$
and
$$\text{(ii)}\ \ \ \ \|\varphi-\varphi|_{F_1}\otimes \varphi|_{F_1^{\text{c}}}\|<\varepsilon.$$
The last inequality follows from the fact, that when $w$ almost commutes with $\xi_{\varphi}$, it almost commutes with $\varphi$ too. 
The properties (i) and (ii) above show that $\varphi$ satisfies the product condition of Connes--Woods \cite{ConnesWoods85} and thus $N$ is an ITPFI factor. 
But it is well-known that $R_{\infty}$ is the only ITPFI factor of type III$_1$ (cf. \cite{ArakiWoods68} and \cite{Connes73}).  
\section{Preliminaries} 
\subsection{Notation}
We use $M,N,\dots$ to denote von Neumann algebras and $\xi,\eta,\ldots$ to denote vectors in a Hilbert space. Let $M$ be a von Neumann algebra. $\mathcal{U}(M)$ denotes the unitary group of $M$. For a faithful normal state $\varphi$ on  $M$, we denote by $\Delta_{\varphi}$ (resp. $J_{\varphi}$) the modular operator (resp. modular conjugation operator) associated with $\varphi$, and the modular automorphism group of $\varphi$ is denoted by $\sigma^{\varphi}$. The norm $\|x\|_{\varphi}=\varphi(x^*x)^{\frac{1}{2}}$ defines the strong operator topology (SOT) on the unit ball of $M$. The centralizer of $\varphi$ is denoted by $M_{\varphi}$. 
\subsection{Connes--Woods' characterization of ITPFI factors}
Recall that a von Neumann algebra $M$ with separable predual is called {\it hyperfinite} if there exists an increasing sequence $M_1\subset M_2\subset \cdots$ of finite-dimensional *-subalgebras such that $M=(\bigcup_{n=1}^{\infty}M_n)''$. A factor $M$ is called an {\it Araki--Woods} factor or an {\it ITPFI} (infinite tensor product of factors of type I) factor, if it is isomorphic to the factor of the form $$\bigotimes_{i\in I}(M_i,\varphi_i),$$
where $I$ is a countable infinite set and each $M_i$ (resp. $\varphi_i$) is a $\sigma$-finite type I factor (resp. a faithful normal state). 
Araki and Woods classified most ITPFI factors: 
\begin{theorem}[{\cite{ArakiWoods68}}]
There exists a unique {\rm{ITPFI}} factor with separable predual for each type {\rm{I}}$_{\infty}$, ${\rm{II}}_1$, ${\rm{II}}_{\infty}$ and ${\rm{III}}_{\lambda}, \lambda \in (0,1]$. In particular, all {\rm{ITPFI}} factors of type {\rm{III}}$_1$ are isomorphic to 
$$R_{\infty}:=\bigotimes_{n\in \mathbb{N}}(M_3(\mathbb{C}), {\rm{Tr}}(\rho\ \cdot\ )),$$
where $\rho:=\frac{1}{1+\lambda+\mu}{\rm{diag}}(1,\lambda,\mu)$ and $0<\lambda,\mu$ satisfies $\frac{\log \lambda}{\log \mu}\notin \mathbb{Q}$. 
\end{theorem}
It is clear that an ITPFI factor with separable predual is hyperfinite. 
The converese is also true for factors not of type III$_0$, but false in general. 
Namely, Connes--Woods \cite{ConnesWoods85} characterized hyperfinite factors of type III$_0$ with separable predual which are  isomorphic to ITPFI factors
 by the approximate transitivity of their flow of weights, while the existence of hyperfinite factors of type III$_0$ with separable predual which are not isomorphic to ITPFI factors had been shown in \cite{Connes72}.
Let $N$ be a von Neumann algebra, and let $F$ be a finite dimensional subfactor of $N$ with relative commutant $F^{\text{c}}:=F'\cap N$ in $N$. Then it is elementary to check, that the map 
$$\sum_{i=1}^nx_i\otimes y_i\mapsto \sum_{i=1}^nx_iy_i,\ \ \ \ \ \ x_i\in F,\ y_i\in F^{\text{c}}\ (1\le i\le n)$$
is an isomorphism of $F\otimes F^{\text{c}}$ onto $N$. If $\omega_1$ is a normal state on $F$ and $\omega_2$ is a normal state on $F^{\text{c}}$, we let $\omega_1\otimes \omega_2$ denote the corresponding state on $N$, i.e., 
\[(\omega_1\otimes \omega_2)(xy)=\omega_1(x)\omega_2(y),\ \ \ \ \ x\in F,\ y\in F^{\text{c}}.\]
In our proof of 
$$[N \text{\ injective III}_1\text{\ and\ }B_{\varphi}=\mathbb{C}1]\Rightarrow N\cong R_{\infty},$$
we shall need the following criterion for a factor to be ITPFI: 
\begin{proposition}[{\cite[Lemma 7.6]{ConnesWoods85}}]\label{prop: CW}
Let $N$ be a factor on a separable Hilbert space. Then $N$ is {\rm{ITPFI}} if and only if $N$ admits a normal faithful state $\varphi$ with the following property: 
for every finite set $x_1,\ldots,x_n$ of operators in $N$, for every $\varepsilon>0$, and every strong* neighborhood $\mathcal{V}$ of $0$ in $N$, there exists a finite dimensional subfactor $F$ of $N$, such that 
$$x_k\in F+\mathcal{V},\ \ \ \ \ k=1,\ldots,n$$
and 
$$\|\varphi-\varphi|_F\otimes \varphi|_{F^{\text{c}}}\|<\varepsilon.$$  
\end{proposition}
\subsection{Bicentralizers on type III$_1$ factors}\label{subsec: bicentralizers}
In this subsection, we recall Connes' bicentralizers. 
Let $M$ be a $\sigma$-finite von Neumann algebra, and let $\varphi$ be a normal faithful state on $M$. 
We denote by ${\rm{AC}}(\varphi)$ the set of all norm-bounded sequences $(x_n)_{n=1}^{\infty}$ in $M$ such that $\lim_{n\to \infty}\|\varphi x_n-x_n\varphi\|=0$ holds. 
\begin{definition}[Connes] 
The {\it bicentralizer} of $\varphi$ is the set $B_{\varphi}$ of all $x\in M$ such that $\lim_{n\to \infty}\|xa_n-a_nx\|_{\varphi}=0$ holds for all $(a_n)_{n=1}^{\infty}\in {\rm{AC}}(\varphi)$. 
\end{definition}
Since $B_{\varphi}$ is a von Neumann subalgebra of $M$ \cite[Proposition 1.3]{Haagerup87}, it holds that $\lim_{n\to \infty}\|xa_n-a_nx\|_{\varphi}^{\sharp}=0$. 

It was conjectured by Connes that for all factors of type III$_1$ with separable predual, the bicentralizer $B_{\varphi}$ of any normal faithful state $\varphi$ on $M$ is trivial, i.e.,  $B_{\varphi}=\mathbb{C}1$ holds. This is still an open problem. 
We will need the following result on type III$_1$ factors, known as the Connes--St\o rmer transitivity:
\begin{theorem}[{\cite{ConnesStormer78}}] Let $M$ be a type {\rm{III}}$_1$ factor with separable predual. Then for every faithful normal states $\varphi,\psi$ on $M$ and $\varepsilon>0$, there exists a unitary $u\in \mathcal{U}(M)$ such that $\|u\varphi u^*-\psi\|<\varepsilon$ holds. 
\end{theorem}
Connes showed that by the Connes-St\o rmer transitivity, for a type III$_1$ factor $M$ with separable predual, the triviality of $B_{\varphi}$ for one fixed faithful normal state $\varphi$ on $M$ implies the triviality of $B_{\psi}$ for every faithful normal state $\psi$ (see \cite[Corollary 1.5]{Haagerup87} for the proof).  He also showed that the triviality of the bicentralizer is equivalent to the following property (the proof is given in \cite[Proposition 1.3 (2)]{Haagerup87}): 
\begin{proposition}[Connes]\label{prop: bicentralizer convex combination}Let $M$ be a von Neumann algebra with a normal faithful state $\varphi$. Then $B_{\varphi}=\mathbb{C}1$ holds, if and only if the following condition is satisfied: for every $a\in M$ and $\delta>0$, 
\begin{equation*}
\overline{\rm{conv}}\{u^*au; u\in \mathcal{U}(M),\ \|u\varphi-\varphi u\|\le \delta\}\cap \mathbb{C}1\neq \emptyset,
\end{equation*}
where $\overline{\rm{conv}}$ is the closure of the convex hull in the $\sigma$-weak topology. 
\end{proposition}
We will use the following variant of Proposition \ref{prop: bicentralizer convex combination}. 
\begin{proposition}\label{prop: bicentralizer condition}
Let $M$ be a type {\rm{III}}$_1$ factor with separable predual, and let $\varphi$ be a normal faithful state on $M$ whose modular automorphism group $\sigma^{\varphi}$ leaves a finite-dimensional subfactor $F$ globally invariant. Let $\varepsilon_{F,\varphi}\colon M\to F$ be the normal faithful $\varphi$-preserving conditional expectation. 
Assume that $B_{\varphi}=\mathbb{C}1$. Then for every $\delta>0$ and $a\in M$, we have 
\begin{equation}
\varepsilon_{F,\varphi}(a)\in \overline{\rm{conv}}\{u^*au;\ u\in \mathcal{U}(F^{\text{c}}),\ \|u\xi_{\varphi}-\xi_{\varphi}u\|\le \delta\}.\label{eq: relative Scwarz2}
\end{equation}
Here, $\xi_{\varphi}$ is the representing vector of $\varphi$ in the natural cone.
\end{proposition}
\begin{proof}
The proof is essentially the same as Proposition \ref{prop: bicentralizer convex combination}, so we only indicate the outline.   
Note that by Araki-Powers-St\o rmer inequality, for every $u\in \mathcal{U}(M)$ one has:
\begin{equation*}
\|\xi_{\varphi}-u\xi_{\varphi}u^*\|^2\le \|\varphi-u\varphi u^*\|\le \|\xi_{\varphi}-u\xi_{\varphi}u^*\|\cdot \|\xi_{\varphi}+u\xi_{\varphi}u^*\|.
\end{equation*}
 Therefore in the arguments below, we may replace the condition $``\|u\xi_{\varphi}-\xi_{\varphi}u\|\le \delta"$ in Proposition \ref{prop: bicentralizer condition} with the condition  $``\|u\varphi-\varphi u\|\le \delta"$, as we take $\delta>0$ to be arbitrarily small. 
 As was pointed out in \cite[Remark 1.4]{Haagerup87}, it follows from the proof of Proposition \ref{prop: bicentralizer convex combination} that the condition $B_{\varphi}=\mathbb{C}1$ is equivalent to the next condition that for all $a\in M$ and $\delta>0$, 
\begin{equation}
\varphi (a)1\in \bigcap_{\delta>0}\overline{\rm{conv}}\{u^*au; u\in \mathcal{U}(M),\,\|u\xi_{\varphi}-\xi_{\varphi}u\|<\delta\}.\label{eq: Rem1.4}
\end{equation}
Let $a\in M$. Since $M\cong F\otimes F^{\text{c}}$ with $\varphi=\varphi|_F\otimes \varphi|_{F^{\text{c}}}$, we may now apply (\ref{eq: Rem1.4}) to $F^{\text{c}}(\cong M)$ and $\varphi|_{F^{\text{c}}}$ to obtain 
\[\varepsilon_{F,\varphi}(a)=\text{id}_F\otimes \varphi|_{F^{\text{c}}}(a)\in \overline{\rm{conv}}\{u^*au;\ u\in \mathcal{U}(F^{\text{c}}),\ \|u\xi_{\varphi}-\xi_{\varphi}u\|\le \delta\}.\]
Note that we used the fact that $\|\varphi u-u\varphi\|=\|\psi u-u\psi\|$, where $\psi:=\varphi|_{F^{\text{c}}}$ and $u\in \mathcal{U}(F^{\text{c}})$ thanks to the existence of a normal faithful $\varphi$-preserving conditional expectation from $M$ onto $F^{\text{c}}$. 
 \end{proof}
\subsection{Almost unitary equivalence in Hilbert $N$-bimodules}
We recall a result about almost unitary equivalence in Hilbert bimodules established in \cite{Haagerup89} which is a generalization of \cite[Theorem 4.2]{Haagerup85}. 
Let $N$ be a von Neumann algbera, and $H$ be a normal Hilbert $N$-bimodule, i.e., $H$ is a Hilbert space on which there are defined left and right actions by elements from $N$: 
\begin{equation*}
(x,\xi)\mapsto x\xi,\ \ \ (x,\xi)\mapsto \xi x,\ \ \ \ x\in N,\ \xi\in H
\end{equation*}
such that the above maps $N\times H\to H$ are bilinear and 
\begin{equation*}
(x\xi)y=x(\xi y),\ \ \ \ \ \ x,y\in N,\ \ \ \xi\in H. 
\end{equation*}
Moreover, $x\mapsto L_x$, where $L_x\xi:=x\xi\ (\xi\in H)$ is a normal unital $*$-homomorphism, and $x\mapsto R_x$, where $R_x\xi:=\xi x\ (\xi\in H)$ is a normal unital $*$-antihomomorphism. 
\begin{definition}
Let $N$ be a von Neumann algebra, let $(N,H)$ be a normal Hilbert $N$-bimodule, and let $\delta\in \mathbb{R}_+$. Two $n$-tuples $(\xi_1,\ldots,\xi_n)$ and $(\eta_1,\ldots,\eta_n)$ of unit vectors in $H$ are called {\it $\delta$-related}, if there exists a family $(a_i)_{i\in I}$ of operators in $N$, such that 
\begin{equation*}
\sum_{i\in I}a_i^*a_i=\sum_{i\in I}a_ia_i^*=1
\end{equation*}
and 
\begin{equation*}
\sum_{i\in I}\|a_i\xi_k-\eta_ka_i\|^2<\delta,\ \ \ \ \ \ \ \ \ k=1,\ldots, n. 
\end{equation*}
\end{definition}
We will use the following result which relates the $\delta$-relatedness to approximate unitary equivalence in Hilbert $N$-bimodules: 
\begin{theorem}[{\cite[Theorem 2.3]{Haagerup89}}]\label{thm: almost unitary equivalence}
 For every $n\in \mathbb{N}$ and $\varepsilon>0$, there exists a $\delta=\delta(n,\varepsilon)>0$, such that for all von Neumann algebra $N$ and $\delta$-related $n$-tuples $(\xi_1,\ldots,\xi_n)$ and $(\eta_1,\ldots,\eta_n)$ of unit vectors in a normal Hilbert $N$-bimodule $H$, there exists a unitary $u\in \mathcal{U}(N)$ such that 
\begin{equation*}
\|u\xi_k-\eta_ku\|<\varepsilon,\ \ \ \ \ \ k=1,\ldots, n. 
\end{equation*}
\end{theorem}
\begin{remark}\label{rem: approximate related}
As can be seen in the proof of \cite[Theorem 2.3]{Haagerup89}, in order to show that the conclusion of Theorem \ref{thm: almost unitary equivalence} holds, it suffices to show the following: for every $\sigma$-strong neighborhood $\mathcal{V}$ of $0$ in $N$, there exist $a_1,\ldots, a_p\in N$ such that 
\begin{align}
\sum_{i=1}^p\|a_i\xi_k-\eta_ka_i\|^2<\delta,\ \ \ \ \ \ k=1,\ldots, n \label{eq: approximate related}\\
\sum_{i=1}^pa_i^*a_i\le 1,\ \ \ \sum_{i=1}^pa_ia_i^*\le 1\label{eq: approximate related2}\\
\sum_{i=1}^pa_i^*a_i\in 1+\mathcal{V},\ \ \ \ \sum_{i=1}^pa_ia_i^*\in 1+\mathcal{V}.\label{eq: approximate related3}
\end{align}
This is because we can obtain the conclusions of \cite[Lemma 2.5]{Haagerup89} out of (\ref{eq: approximate related}), (\ref{eq: approximate related2}) and (\ref{eq: approximate related3}), which is enough to prove Theorem \ref{thm: almost unitary equivalence}.  We will use this variant in the proof Lemma \ref{lem: UH5.5}. 
\end{remark}
\section{Completely positive maps from $m\times m$-matrices into an injective factor of type III$_1$}
The main result of this section is:
\begin{theorem}\label{thm: UH3.1}
Let $N$ be an injective factor of type {\rm{III}}$_1$ with separable predual, and let $\varphi$ be a faithful normal state on $N$. Then for every finite set $u_1,\ldots,u_n$ of unitaries in $N$, and every $\varepsilon,\delta>0$, there exists $m\in \mathbb{N}$, a unital completely positive map $T\colon M_m(\mathbb{C})\to N$, and $n$ unitaries $v_1,\ldots,v_n$ in $M_m(\mathbb{C})$, such that  $\psi=\varphi\circ T$ is a normal faithful state on $M_m(\mathbb{C})$, and 
\eqa{
\|\sigma_t^{\varphi}\circ T-T\circ \sigma_t^{\psi}\|&\le \delta |t|,\ \ \ t\in \mathbb{R},\\
\|T(v_k)-u_k\|_{\varphi}&<\varepsilon,\ \ \ k=1,\ldots,n.
} 
\end{theorem}
In the following we let $M=N\rtimes_{\sigma^{\varphi}}\mathbb{R}$ be the crossed product of $N$ by $\sigma^{\varphi}$ with generators $\pi_{\sigma^{\varphi}}(x)\ (x\in N)$ and $\lambda (s)\ (s\in \mathbb{R})$. We identify $\pi_{\sigma^{\varphi}}(x)$ with $x\in N$. Let $a$ be the (unbounded) self-adjoint operator for which $\lambda(s)=\exp (isa)\ (s\in \mathbb{R})$. 
For $f\in L^1(\mathbb{R})$, we define the Fourier transform $\hat{f}$ by 
\begin{equation*}
\hat{f}(s)=\frac{1}{\sqrt{2\pi}}\int_{-\infty}^{\infty}e^{-ist}f(t)\,dt,\ \ \ \ \ s\in \mathbb{R}.
\end{equation*}
In the sequel, von Neumann algebra-valued integrals are understood to be the $\sigma$-weak sense. 
%That is, if $x\colon \mathbb{R}\to M$ is a $\sigma$-weakly measurable map and $\mu$ is a $\sigma$-finite positive Borel measure on $\mathbb{R}$, such that $t\mapsto \|x(t)\|$ is in $L^1(\mathbb{R},\mu)$, then $\int_{-\infty}^{\infty}x(t)\,d\mu (t)$ denotes the unique element $y\in M$ such that 
%\[\psi(y)=\int_{-\infty}^{\infty}\psi(x(t))\,d\mu(t),\ \ \ \ \ \psi\in M_*.\] 
Let $(\theta^{\varphi}_s)_{s\in \mathbb{R}}$ be the dual action of $\sigma^{\varphi}$ on $M$. 
By \cite{Haagerup79-2}, there exists a normal faithful semifinite operator-valued weight $P\colon M_+\to \widehat{N}_+$ ($\widehat{N}_+$ is the extended positive part of $N$)
given by  
\begin{equation}
P(x)=\int_{-\infty}^{\infty}\theta_s^{\varphi}(x)\,ds,\ \ \ \ \ x\in M_+.\label{eq: operator valued weight}
\end{equation}
Following \cite{ConnesTakesaki77}, if we put 
\begin{equation*}
\mathfrak{m}:=\text{span}\left \{x\in M_+; \sup_{c>0}\left \|\int_{-c}^c\theta_t^{\varphi}(x)\,dt\right \|<\infty\right \},
\end{equation*} 
then the formula (\ref{eq: operator valued weight}) for $x\in \mathfrak{m}$ makes sense and $P(x)\in N$. Moreover, $\mathfrak{m}\ni x\mapsto P(x)\in N$ defines a positive linear map.  

For all $x\in \mathfrak{m}$, the $\sigma$-weak integral $\int_{-c}^c\theta_t^{\varphi}(x)\,dt$ is $\sigma$-strongly convergent as $c\to \infty$. The range of $P$ is contained in $\pi_{\sigma^{\varphi}}(N)$, because  $\pi_{\sigma^{\varphi}}(N)$ is the fixed point algebra in $M$ under the dual action. 

\begin{lemma}\label{lem: UH3.2}
Let $t\mapsto x(t)$ be a $\sigma$-strongly* continuous function from $\mathbb{R}$ to $N$ such that $t\mapsto \|x(t)\|$ is in $L^1(\mathbb{R})\cap L^{\infty}(\mathbb{R})$. Put
\[x:=\int_{-\infty}^{\infty}\lambda (t)x(t)\,dt\in M.\]
Then $x^*x\in \mathfrak{m}$, and 
\begin{equation*}
P(x^*x)=2\pi \int_{-\infty}^{\infty}x(t)^*x(t)\,dt.\label{eq: P(x^*x)}
\end{equation*}

\end{lemma}
 \begin{proof}
 Note first, that
 \begin{align*}
x^*x&=\iint_{\mathbb{R}^2}x(s)^*\lambda(t-s)x(t)\,dsdt\notag \\
&=\iint_{\mathbb{R}^2}x(s)^*\lambda(t)x(s+t)\,dsdt.
\end{align*}
Put $f_n(s)=e^{-s^2/(4n)}\ (s\in \mathbb{R})$, and 
\begin{align*}
g_n(t)=\frac{1}{2\pi}\int_{-\infty}^{\infty}f_n(s)e^{-its}\,ds=\left (\frac{n}{\pi
}\right )^{\frac{1}{2}}e^
{-nt^2}\ \ (t\in \mathbb{R}).
\end{align*}
Using that $\theta_s^{\varphi}(y)=y\ (s\in \mathbb{R}, y\in N)$, $\theta_s^{\varphi}(\lambda(t))=e^{-ist}\lambda(t)$\ $(s,t\in \mathbb{R})$\footnote{Typewriter's note: Haagerup used the convention $\theta_s^{\varphi}(\lambda(t))=e^{ist}\lambda(t)$. However, since the negative sign convention is widely accepted, we decided to change the definition.} and the Fubini Theorem, we have for every $\psi\in M_*$,
\begin{align*}
\nai{\psi}{\int_{-\infty}^{\infty}\theta_u^{\varphi}(x^*x)f_n(u)\,du}&=\int_{-\infty}^{\infty}\int_{-\infty}^{\infty}\int_{-\infty}^{\infty}e^{-itu}f_n(u)\psi(x(s)^*\lambda(t)x(s+t))\,dsdtdu\notag \\
&=\int_{-\infty}^{\infty}g_n(t)\left (\int_{-\infty}^{\infty}2\pi \psi(x(s)^*\lambda(t)x(s+t))ds\right )\,dt.
%&=2\pi \int_{-\infty}^{\infty}\left (\int_{-\infty}^{\infty}x(s)^*\textcolor{red}{\lambda(t)}x(s+t)ds\right )g_n(t)dt.
\end{align*}

Since $t\mapsto \int_{-\infty}^{\infty}\psi(x(s)^*\lambda(t)x(s+t))ds$ is in $C_0(\mathbb{R})$ and $g_n\stackrel{n\to \infty}{\to}\delta_0$ (weak$^*$ in $C_0(\mathbb{R})$), we have 
%$(g_n)_{n=1}^{\infty}$ forms an approximate unit on $\mathbb{R}$, we have  
\begin{equation*}
\lim_{n\to \infty}\nai{\psi}{\int_{-\infty}^{\infty}\theta_u^{\varphi}(x^*x)f_n(u)\,du}=\nai{\psi}{2\pi \int_{-\infty}^{\infty}x(s)^*x(s)\,ds}.
\end{equation*}
Since $\psi\in M_*$ is arbitrary, $\theta_n(x^*x)\ge 0$ and $f_n\nearrow 1$ uniformly on compact sets, it follows that 
\begin{equation*}
\lim_{n\to \infty}\int_{-\infty}^{\infty}\theta_u^{\varphi}(x^*x)f_n(u)du=2\pi \int_{-\infty}^{\infty}x(s)^*x(s)\,ds\ \ \ \ (\sigma{\rm{-strongly}}).
\end{equation*}
Therefore $x^*x\in \mathfrak{m}$, and $\displaystyle P(x^*x)=2\pi \int_{-\infty}^{\infty}x(t)^*x(t)\,dt.$ 
 \end{proof}
\begin{lemma}\label{lem: UH3.3}
Let $a$ be the (unbounded) self-adjoint operator affiliated with $M$ for which $\exp (ita)=\lambda (t)\ (t\in \mathbb{R})$ holds. 
Let $\alpha>0$, and let $e_{\alpha}$ be the spectral projection of the operator $a$ corresponding to the interval $[0,\alpha]$. Then for each $x\in N$, one has $e_{\alpha}xe_{\alpha}\in \mathfrak{m}$ and 
\begin{equation}
P(e_{\alpha}xe_{\alpha})=\int_{-\infty}^{\infty}\sigma_t^{\varphi}(x)\frac{1-\cos \alpha t}{\pi t^2}\,dt,\ \ \ \ \ x\in N.\label{eq: UH3.3}
\end{equation}
\end{lemma}
\begin{proof}
It is sufficient to consider the case $x\ge 0$, so we can assume that $x=y^*y\ (y\in N)$. 
For $f\in L^1(\mathbb{R})\cap L^{\infty}(\mathbb{R})\cap C(\mathbb{R})$, we have
\begin{align*}
y\hat{f}(a)=\frac{1}{\sqrt{2\pi}}\int_{-\infty}^{\infty}y\lambda(-t)f(t)dt&=\frac{1}{\sqrt{2\pi}}\int_{-\infty}^{\infty}\lambda(-t)\sigma_{t}^{\varphi}(y)f(t)\,dt\notag\\
&=\frac{1}{\sqrt{2\pi}}\int_{-\infty}^{\infty}\lambda(t)\sigma_{-t}^{\varphi}(y)f(-t)\,dt.
\end{align*}
Hence by Lemma \ref{lem: UH3.2}, $\hat{f}(a)^*x\hat{f}(a)\in \mathfrak{m}$, and 
\begin{equation*}
P(\hat{f}(a)^*x\hat{f}(a))=\int_{-\infty}^{\infty}\sigma_{t}^{\varphi}(x)|f(t)|^2\,dt.
\end{equation*}
For $n>\frac{2}{\alpha}$, let $g_n$ be the continuous function on $\mathbb{R}$ for which 
\begin{align*}
g_n(t)&=0,\ \ \ \ t\le 0,\ t\ge \alpha,\\
g_n(t)&=1,\ \ \ \ t\in [\tfrac{1}{n},\alpha-\tfrac{1}{n}],
\end{align*}
and for which the graph is a straight line on $[0,\frac{1}{n}]$ and $[\alpha-\frac{1}{n},\frac{1}{n}]$. Since $a$ has no point spectrum, 
$g_n(a)\nearrow e_{\alpha}\ (n\to \infty)$. It is elementary to check that each $g_n$ is of the form $g_n=\hat{f}_n$ for a function $f_n\in L^1(\mathbb{R})\cap L^{\infty}(\mathbb{R})\cap C(\mathbb{R})$ (use, for example, the fact that $g_n=n1_{[0,\tfrac{1}{n}]}*1_{[0,\alpha-\frac{1}{n}]}$). Hence $g_n(a)^2\in \mathfrak{m}$, and by the Plancherel Theorem, we get 
\begin{equation*}
P(g_n(a)^2)=P(\hat{f}_n(a)^*\hat{f}_n(a))=\|f_n\|_2^21=\|\hat{f}_n\|_2^21. 
\end{equation*}
Since $\sup_n \|\hat{f}_n\|_2^2=\alpha<\infty$, we have $e_{\alpha}\in \mathfrak{m}$ and $P(e_{\alpha})=\alpha 1$. Therefore $e_{\alpha}Me_{\alpha}\subseteq \mathfrak{m}$, and the restriction of $P$ to $e_{\alpha}Me_{\alpha}$ is a positive normal map. Hence for $x\in N$, 
\begin{equation*}
P(e_{\alpha}xe_{\alpha})=\lim_{n\to \infty}P(g_n(a)xg_n(a))=\lim_{n\to \infty}\int_{-\infty}^{\infty}\sigma_{t}^{\varphi}(x)|f_n(t)|^2\,dt\ \ \ (\sigma \text{-strongly}).
\end{equation*}
Since $\|g_n-1_{[0,\alpha]}\|_2\stackrel{n\to \infty}{\to}0$, it follows that $f_n$ converges in $L^2(\mathbb{R})$ to the function 
\begin{align*}
f(t)&=\frac{1}{\sqrt{2\pi}}\int_{-\infty}^{\infty}1_{[0,\alpha]}(s)e^{ist}\,ds\notag \\
&=-\frac{i}{t\sqrt{2\pi}}(e^{i\alpha t}-1).
\end{align*}
Hence $|f_n|^2\stackrel{n\to \infty}{\to}|f|^2$ in $L^1(\mathbb{R})$, with $|f|^2(t)=\frac{1}{\pi t^2}(1-\cos \alpha t)\ (t\in \mathbb{R})$. Therefore (\ref{eq: UH3.3}) holds. 
%\begin{equation*}
%P(e_{\alpha}xe_{\alpha})=\int_{-\infty}^{\infty}\sigma_{t}^{\varphi}(x)\frac{1}{\pi t^2}(1-\cos \alpha t)dt.\label{eq: P(exe) last}
%\end{equation*}
\end{proof}

\begin{lemma}\label{lem: UH3.4}
Let $N$ be an injective factor of type {\rm{III}}$_1$ with separable predual, $\varphi $ be a faithful normal state on $N$ and let $R$ be the hyperfinite {\rm{II}}$_1$ factor with tracial state $\tau$. For every finite set $u_1,\ldots,u_n$ of unitaries in $N$ and every $\varepsilon>0$, there exist $x_1,\ldots,\ x_n$ in the unit ball of $R$, a normal unital completely positive map $T\colon R\to N$, such that $\psi=\varphi\circ T$ is a normal faithful state on $R$, and 
\begin{align}
\sigma_t^{\varphi}\circ T&=T\circ \sigma_t^{\psi},\ \ \ \ \ \ \ t\in \mathbb{R},\label{eq: 3.4-1}\\
\|T(x_k)-u_k\|_{\varphi}&<\varepsilon,\ \ \ \ \ \ \ \ \ \ k=1,\ldots,n\label{eq: 3.4-2}.
\end{align} 
Moreover, the spectrum of $h={\rm{d}}\psi/{\rm{d}}\tau$ is a closed interval $[\lambda_1,\lambda_2],\ 0<\lambda_1<\lambda_2<\infty$, and $h$ has no eigenvalues.  
\end{lemma}
\begin{proof}\footnote{Typewriter's note: Since some pages were missing from the original notes, we could not find all parts of the proofs of Lemma \ref{lem: UH3.4} and Theorem \ref{thm: UH3.1}. We include the following proof for the reader's convenience.}
Let $M=N\rtimes_{\sigma^{\varphi}}\mathbb{R}$. 
By \cite{Takesaki73-2}, $M$ has a normal faithful semifinite trace $\tau$, such that 
\begin{equation*}
\tau\circ \theta_s^{\varphi}=e^{-s}\tau\ \ (s\in \mathbb{R}).
\end{equation*}
The trace $\tau$ can be constructed in the following way: Let $\tilde{\varphi}$ be the dual weight of $\varphi$ on $M$ (cf. \cite{Haagerup79}). Let $a$ be the self-adjoint operator for which $\exp (ita)=\lambda (t)\ (t\in \mathbb{R})$. Then $a$ is affiliated with the centralizer $M_{\tilde{\varphi}}$ of $\tilde{\varphi}$ and 
\begin{equation*}
\tau=\tilde{\varphi}(e^{-a}\ \cdot\ )
\end{equation*}
in the sense of Pedersen-Takesaki \cite{PedersenTakesaki73}. By \cite{Haagerup79-2}, $\tilde{\varphi}$ is on the subspace $\mathfrak{m}$ given by 
\begin{equation*}
\tilde{\varphi}(x)=\varphi\circ P(x),\ \ \ \ \ x\in \mathfrak{m}.
\end{equation*}
Let $\alpha>0$, and let $e_{\alpha}=1_{[0,\alpha]}(a)$. Then by Lemma \ref{lem: UH3.3}, $e_{\alpha}\in \mathfrak{m}$ and $P(e_{\alpha})=\alpha 1$. Hence $\tilde{\varphi}|_{e_{\alpha}Me_{\alpha}}$ is a positive normal functional, and $\tilde{\varphi}(e_{\alpha})=\alpha$. Finally, 
\begin{equation*}
\tau(e_{\alpha})=\tilde{\varphi}(e^{-a}e_{\alpha})=\int_0^{\alpha}e^{-t}\,dt=1-e^{-\alpha}<\infty,
\end{equation*}
because
\begin{equation*}
e^{-a}e_{\alpha}=\int_0^{\alpha}e^{-t}\,de(\lambda),
\end{equation*}
where 
\begin{equation*}
a=\int_{-\infty}^{\infty}\lambda\,de(\lambda)
\end{equation*}
is the spectral resolution of $a$, and $d\tilde{\varphi}(e(\lambda))=d\lambda$. 

Since $N$ is of type III$_1$, $M$ is a type II$_{\infty}$ factor, and therefore $e_{\alpha}Me_{\alpha}$ is a II$_1$ factor. Moreover, the injectivity of $N$ implies that $M$ is also injective, so that $e_{\alpha}Me_{\alpha}$ is isomorphic to the hyperfinite factor $R$ of type II$_1$ by \cite{Connes76}.\\ \\
\textbf{Claim.} For any $x\in N$, we have 
\begin{equation*}
\lim_{\alpha \to \infty}\left \|\frac{1}{\alpha}P(e_{\alpha}xe_{\alpha})-x\right \|_{\varphi}=0. 
\end{equation*}
This follows from a basic property of the Fej\'er kernel (see e.g., \cite[Chapters I and VI]{Katznelson68}), but we include the proof for completeness. Let $\varepsilon>0$. Choose $t_0>0$ small enough so that $\|\sigma_t^{\varphi}(x)-x\|_{\varphi}\le \varepsilon$ for all $t\in [-t_0,t_0]$. Moreover, by $\displaystyle \left |\frac{1-\cos (\alpha t)}{\pi \alpha t^2}\right |\le \frac{2}{\pi \alpha}\cdot \frac{1}{t^2}$, we have 
\begin{equation*}
\lim_{\alpha \to \infty}\int_{|t|\ge t_0}\frac{1-\cos (\alpha t)}{\pi \alpha t^2}\,dt=0.
\end{equation*} 
By Lemma \ref{lem: UH3.3}, we have 
\begin{align*}
\limsup_{\alpha \to \infty}\left \|\frac{1}{\alpha}P(e_{\alpha}xe_{\alpha})-x\right \|_{\varphi}&\le \limsup_{\alpha \to \infty}\int_{-\infty}^{\infty}\|\sigma_t^{\varphi}(x)-x\|_{\varphi}\frac{1-\cos (\alpha t)}{\pi \alpha t^2}\,dt\notag \\
&\le \varepsilon+2\|x\|_{\varphi}\limsup_{\alpha \to \infty}\int_{|t|\ge t_0}\frac{1-\cos (\alpha t)}{\pi \alpha t^2}\,dt\notag \\
&=\varepsilon.
\end{align*}
Since $\varepsilon>0$ is arbitrary, we obtain the conclusion.\\
\medskip 
Let $n\ge 1$, $u_1,\ldots,u_n\in \mathcal{U}(N)$ and $\varepsilon,\delta>0$ be given. By the above claim, we may choose $\alpha>0$ large enough so that 
\begin{equation}
\left \|\frac{1}{\alpha}P(e_{\alpha}u_ke_{\alpha})-u_k\right \|_{\varphi}<\varepsilon,\ \ \ \ \ \ 1\le k\le n.\label{eq: alpha^{-1}P}
\end{equation}
Define $T:=\alpha^{-1} P |_{e_\alpha M e_\alpha}\colon R=e_{\alpha}Me_{\alpha}\to N$ and 
\begin{equation*}
\psi:=\varphi \circ T=\frac{1}{\alpha}\varphi\circ P(e_{\alpha} \cdot e_{\alpha})=\frac{1}{\alpha}\tilde{\varphi}(e_{\alpha} \cdot e_{\alpha}).
\end{equation*}
Then $T$ is a normal unital completely positive map, and $\psi$ is a normal faithful state on $R$. By (\ref{eq: alpha^{-1}P}) we have 
\begin{equation*}
\|T(x_k)-u_k\|_{\varphi}<\varepsilon,\ \ \ \ \ 1\le k\le n,
\end{equation*}
where $x_k:=e_{\alpha}u_ke_{\alpha}\ (1\le k\le n)$ are in the unit ball of $R$. 
Moreover, since $e_{\alpha}\in M_{\tilde{\varphi}}$ and since $\sigma_t^{\varphi}\circ P=P\circ \sigma_t^{\tilde{\varphi}}\ (t\in \mathbb{R})$, we have 
$\sigma_t^{\varphi}\circ T=T\circ \sigma_t^{\psi}\ (t\in \mathbb{R})$. By construction, we have 
\begin{equation*}
h:=\frac{\text{d}{\psi}}{\text{d}\tau}=\frac{1-e^{-\alpha}}{\alpha}\exp (a)e_{\alpha},
\end{equation*} which has no atoms and the spectrum of $h$ is a closed bounded interval in $\mathbb{R}_+^*=(0,\infty)$. 
\end{proof}

\begin{lemma}\label{lem: inclusion of spectrum}
Let $B\subset A$ be an inclusion of unital ${\rm{C}}^{\ast}$-algebras and $E\colon A\to B$ be a unital completely positive map. Let $h\in A$ be a self-adjoint element with $\sigma(h)\subset [\lambda_1,\lambda_2]$, where $\sigma(\ \cdot\ )$ denotes the spectrum and $\lambda_1<\lambda_2$ are reals.  
Then $\sigma(E(h))\subset [\lambda_1,\lambda_2]$. 
\end{lemma}
\begin{proof}
Let $\lambda<\lambda_1$. Then $h-\lambda$ is positive and invertible. Take a nonzero $x\in A$ such that $(h-\lambda)^{\frac{1}{2}}x=1$, so that 
$E((h-\lambda)^{\frac{1}{2}}xx^*(h-\lambda)^{\frac{1}{2}})=1$. The left hand side is dominated by $\|x\|^2E(h-\lambda)$, whence 
$E(h-\lambda)\ge \|x\|^{-2}1$, showing that $E(h)-\lambda 1$ is invertible. Thus $\lambda\notin \sigma(E(h))$. Similarly, $\sigma(E(h))\cap (\lambda_2,\infty)=\emptyset$ holds. Therefore $\sigma(E(h))\subset [\lambda_1,\lambda_2]$.  
\end{proof}
\begin{proof}[Proof of Theorem \ref{thm: UH3.1}]
We may assume that $0<\varepsilon<1$. 
By Lemma \ref{lem: UH3.4}, there exist a normal unital completely positive map $T\colon R\to N$ and $x_1,\ldots, x_n$ in the unit ball of $R$ satisfying 
$\sigma_t^{\varphi}\circ T=T\circ \sigma_t^{\psi},\ (t\in \mathbb{R})$, where $\psi:=\varphi\circ T\colon R\to N$ and 
\begin{equation}
\|T(x_k)-u_k\|_{\varphi}<\frac{\varepsilon^2}{16},\ \ \ \ \ 1\le k\le n.\label{eq: 0.1}
\end{equation}
Let $h:=\text{d}\psi/\text{d}\tau\in R_+$. Then by Lemma \ref{lem: UH3.4}, $\sigma(h)=[\lambda_1,\lambda_2]$ for some positive reals $\lambda_1<\lambda_2$ and $h$ does not have a point spectrum. 
Since $\log (\ \cdot\ )$ is continuous on $[\lambda_1,\lambda_2]$, continuous functional calculus guarantees that there exists $\delta'>0$ such that for all $a,b\in R_+$, we have the following implication
\begin{equation}
\sigma(a),\sigma(b)\subset [\lambda_1,\lambda_2]\ \text{and}\ \|a-b\|< \delta'\Rightarrow \|\log a-\log b\|<\frac{\delta}{4}.\label{eq: continuity of log}
\end{equation} 
By using the spectral decomposition of $h$, we can choose a partition of unity $\{p_i\}_{i=1}^{\ell}$ in $R$ and $\{\mu_i\}_{i=1}^{\ell}$ in $\mathbb{R}_+^*$ such that 
\begin{align*}
\tau(p_i)=\frac{1}{\ell},\,\,\, hp_i=p_ih,\\
\|(\log h)p_i-(\log \mu_i)p_i\|<\tfrac{1}{4}\delta,\\
\|hp_i-\mu_ip_i\|<\delta',
\end{align*}
for all $(1\le i\le \ell)$. Let $h_0:=\sum_{i=1}^{\ell}\mu_ip_i$,\, and we have
\begin{equation}
\|h-h_0\|<\delta'\ \ \text{and\ \ }\|\log h-\log h_0\|<\frac{1}{4}\delta.\label{eq: h-h_01}
\end{equation}
Moreover, we may arrange $\{\mu_i\}_{i=1}^{\ell}$ so that  $h_0=\sum_{i=1}^{\ell}\mu_ip_i$ satisfies 
\begin{equation}
\sigma(h_0)\subset [\lambda_1,\lambda_2].\label{eq: estimate on h-h_0}
\end{equation}
Since $R$ is hyperfinite, there exists a type I subfactor $F$ of $R$ so that 
$p_i\in F\ (1\le i\le \ell)$ and 
\begin{equation}
\|x_k-E_F(x_k)\|_{\varphi}<\frac{\varepsilon^2}{16},\ \ \ 1\le k\le n,\label{eq: 0.3}
\end{equation}
where $E_F\colon R\to F$ denotes the $\tau$-preserving conditional expectation. Put $T_F:=T|_F\colon F\to N$ and $y_k:=E_F(x_k)\ (1\le k\le n)$. 
Combining (\ref{eq: 0.1}) and (\ref{eq: 0.3}), for all $1\le k\le n$, we have (use the Schwarz inequality for completely positive maps) 
\begin{align}
\|T_F(y_k)-u_k\|_{\varphi}&\le \|T(y_k)-T(x_k)\|_{\varphi}+\|T(x_k)-u_k\|_{\varphi}\notag \\
&\le \|y_k-x_k\|_{\psi}+\|T(x_k)-u_k\|_{\varphi}\notag \\
&<\frac{\varepsilon^2}{8}.\label{eq: T_Fy_k-u_k}
\end{align}
Then we follow the argument of \cite[Lemma 6.2]{Haagerup89}.  Take $v_1,\ldots,v_k\in \mathcal{U}(F)$ such that 
\begin{equation*}
y_k=v_k|y_k|,\ \ \ \ |y_k|:=(y_k^*y_k)^{\frac{1}{2}},\ \ \ 1\le k\le n.
\end{equation*} 
Then again by the Schwarz inequality for completely positive maps and (\ref{eq: T_Fy_k-u_k}), 
\begin{equation*}
\|y_k\|_{\psi}\ge \|T_F(y_k)\|_{\varphi}>\|u_k\|_{\varphi}-\frac{\varepsilon^2}{8}.
\end{equation*}
Since $\|(y_k^*y_k)^{\frac{1}{2}}\|_{\psi}=\|y_k\|_{\psi}$ and $|y_k|^2+(1-|y_k|)^2\le 1$ (because $0\le |y_k|\le 1$), we have 
\begin{align*}
\|v_k-y_k\|_{\psi}^2&=\|1-|y_k|\|_{\psi}^2\le 1-\|\,|y_k|\,\|_{\psi}^2\notag \\
&<1-(1-\tfrac{\varepsilon^2}{8})^2\notag \\
&<\frac{\varepsilon^2}{4}.
\end{align*}
Therefore since $\varepsilon^2<\varepsilon$, 
\begin{align*}
\|T_F(v_k)-u_k\|_{\varphi}&\le \|T_F(v_k-y_k)\|_{\varphi}+\|T_F(y_k)-u_k\|_{\varphi}\notag \\
&< \|v_k-y_k\|_{\psi}+\frac{\varepsilon^2}{8}\notag \\
&<\varepsilon.
\end{align*}
Next, set $\chi:=\tau(h_0\ \cdot\ )\in (R_*)_+$. Note that $\sigma_t^{\chi|_F}=\sigma_t^{\chi}|_F\ (t\in \mathbb{R})$, since $h_0\in F$. 
Then by (\ref{eq: h-h_01}), we have \begin{align}
\|h^{it}-h_0^{it}\|&=\left \|\int_0^1\frac{\rm{d}}{{\rm{d}}s}e^{ist\log h}e^{i(1-s)t\log h_0}\,ds\right \|\notag \\
&\le \int_0^1\|e^{ist\log h}(t\log h-t\log h_0)e^{i(1-s)t\log h_0}\|\,ds\notag \\
&\le \|\log h-\log h_0\|\,|t|\notag\\
&\le \frac{\delta |t|}{4}.\label{eq: hit-h_0it}
\end{align}
On the other hand, $h_F:={\rm{d}}\psi|_F/{\rm{d}}\tau|_F\in F_+$ is equal to $E_F(h)$. Therefore by Lemma \ref{lem: inclusion of spectrum}, $\sigma(h_F)\subset [\lambda_1,\lambda_2]$. 
Moreover, since $E_F(h_0)=h_0$, we have 
\begin{equation*}
\|h_F-h_0\|=\|E_F(h-h_0)\|\le \|h-h_0\|<\delta'. 
\end{equation*}
This shows by (\ref{eq: continuity of log}) and (\ref{eq: estimate on h-h_0}) that $\|\log h_F-\log h_0\|<\frac{\delta}{4}$. Therefore by the same argument, we have 
\begin{equation}
\|h_F^{it}-h_0^{it}\|\le \frac{\delta |t|}{4},\ \ \  t\in \mathbb{R}.\label{eq: h_Fit-h_0it}
\end{equation}
For all $t\in \mathbb{R}$ and $x\in F$, 
\eqa{
\|\sigma_t^{\varphi}\circ T_F(x)-T_F\circ \sigma_t^{\psi|_F}(x)\|&=\|T(\sigma_t^{\psi}(x)-\sigma_t^{\psi|_F}(x))\|\\
&\le \|\sigma_t^{\psi}(x)-\sigma_t^{\psi|_F}(x)\|\\
&\le \|\sigma_t^{\psi}(x)-\sigma_t^{\chi}(x)\|+\|\sigma_t^{\chi|_F}(x)-\sigma_t^{\psi|_F}(x)\|.
}
By (\ref{eq: hit-h_0it}), we have 
\eqa{
\|\sigma_t^{\psi}(x)-\sigma_t^{\chi}(x)\|&=\|h^{it}xh^{-it}-h_0^{it}xh_0^{-it}\|\\
&\le (\|h^{it}-h_0^{it}\|+\|h^{-it}-h_0^{-it}\|)\|x\|\\
&\le \frac{1}{2}\delta |t|\|x\|.
}
Similarly, by (\ref{eq: h_Fit-h_0it}), 
\eqa{
\|\sigma_t^{\chi|_F}(x)-\sigma_t^{\psi|_F}(x)\|&\le (\|h_0^{it}-h_F^{it}\|+\|h_0^{-it}-h_F^{-it}\|)\|x\|\\
&\le \frac{1}{2}\delta |t|\|x\|.
}
These altogether imply that $\|\sigma_t^{\varphi}\circ T_F(x)-T_F\circ \sigma_t^{\psi|_F}(x)\|\le \delta |t|\|x\|$. 
\end{proof}
\section{$\mathbb{Q}$-stable states on III$_1$ factors}\label{section: Q-stable state}
For technical reasons we shall consider a special class of normal faithful states, which we call $\mathbb{Q}$-stable states, because they have nice properties with respect to certain operations involving rationals. 
\begin{definition}\label{def: UH4.1}
A normal faithful state $\varphi$ on a von Neumann algebra $N$ is called {\it $\mathbb{Q}$-stable}, if for every $m\in \mathbb{N}$, there exist $m$ isometries $u_1,\dots, u_m\in N$ with orthogonal range projections, such that 
\begin{align*}
\sum_{i=1}^mu_iu_i^*&=1,\\
\varphi u_i=\frac{1}{m}&u_i\varphi,\ \ \ \ \ \ \ i=1,\ldots,m. 
\end{align*}
\end{definition}

\begin{theorem}\label{thm: UH4.2}
Every factor of type {\rm{III}}$_1$ with separable predual has a $\mathbb{Q}$-stable normal faithful state.
\end{theorem}
For the proof of Theorem \ref{thm: UH4.2}, we shall need two lemmas: 
\begin{lemma}\label{lem: UH4.3} 
The Araki--Woods factor $R_{\infty}$ has a $\mathbb{Q}$-stable normal faithful state. 
\end{lemma}
\begin{proof}
Let $R_{\lambda}\ (0<\lambda<1)$ be the Powers factor of type III$_{\lambda}$, and let $\varphi_{\lambda}$ be the product state on $R_{\lambda}$. Then $\varphi_{\lambda}$ is normal and faithful, and $\sigma^{\varphi_{\lambda}}$ has period $-2\pi /\log \lambda$. Then the centralizer $(R_{\lambda})_{\varphi_{\lambda}}$ is a type II$_1$ factor (cf. \cite[Th\'eor\`eme 4.2.6]{Connes73}), and there exists an isometry $u\in R_{\lambda}$ such that    
\begin{equation*}
\sigma_t^{\varphi_{\lambda}}(u)=\lambda^{it}u,\ \ \ \ \ \ t\in \mathbb{R}.
\end{equation*}
This implies that $\sigma_t^{\varphi_{\lambda}}(uu^*)=uu^*\ (t\in \mathbb{R})$, i.e., $uu^*\in (R_{\lambda})_{\varphi_{\lambda}}$. Moreover, by \cite[Lemma 1.6]{Takesaki73}, we have  
\begin{equation*}
\varphi_{\lambda}u=\lambda u\varphi_{\lambda},
\end{equation*}
and hence $\varphi_{\lambda}(uu^*)=(\varphi_{\lambda}u)(u^*)=\lambda \varphi_{\lambda}(u^*u)=\lambda.$ 

Assume now, that $\lambda=\frac{1}{m},\ m\in \mathbb{N}, m\ge 2$. Then we can choose $m$ equivalent orthogonal projections $p_1,\ldots, p_m\in (R_{\lambda})_{\varphi_{\lambda}}$ with sum 1, such that $p_1=uu^*$. Next, choose partial isometries $v_1,\ldots,v_m\in (R_{\lambda})_{\varphi_{\lambda}}$ such that 
\begin{equation*}
v_i^*v_i=p_1,\ \ v_iv_i^*=p_i,\ \ \ \ i=1,\ldots,m. 
\end{equation*}
Put $u_i=v_iu,\ \ i=1,\ldots,m.$ Then $u_1,\ldots,u_m$ are $m$ isometries in $R_{\lambda}$, such that $\sum_{i=1}^mu_iu_i^*=1$, and $\varphi_{\lambda}u_i=\lambda u_i\varphi_{\lambda},\ i=1,\ldots,m$. Put now
\begin{equation*}
(P,\varphi)=\bigotimes_{m=2}^{\infty}(R_{\frac{1}{m}},\varphi_{\frac{1}{m}}). 
\end{equation*}
Then it is clear from the above computations, that $\varphi$ is a $\mathbb{Q}$-stable normal faithful state on $P$ (observe that it is sufficient to consider $m\ge 2$ case in Definition \ref{def: UH4.1}). Moreover, $P$ is an ITPFI factor for which the the asymptotic ratio set $r_{\infty}(P)$ contains $\{\frac{1}{m};m\in \mathbb{N}\}$. Since $r_{\infty}(P)\cap \mathbb{R}_+$ is a closed subgroup of $\mathbb{R}_+$, we have $r_{\infty}(P)\supseteq \mathbb{R}_+$. Therefore by Araki--Woods' Theorem \cite[Theorem 7.6]{ArakiWoods68}, $P\cong R_{\infty}$ holds.   
\end{proof}
\begin{lemma}\label{lem: UH4.4}
Let $N$ be a factor of type {\rm{III}}$_1$ with separable predual. Then there exists a normal faithful conditional expectation of $N$ onto a subfactor $P$ isomorphic to $R_{\infty}$. 
\end{lemma}
\begin{proof}
\footnote{Typewriter's note: this result has been extended by Haagerup--Musat \cite[Theorem 3.5]{HM09}, where the authors study more general embeddings of ITPFI type III factors into type III factors as the range of normal faithftul conditional expectations.}
We can write $R_{\infty}$ as an infinite tensor product 
\begin{equation*}
R_{\infty}=\bigotimes_{k=1}^{\infty}(P_k,\omega_k),
\end{equation*}
where each $P_k$ is a copy of the $2\times 2$ matrices $M_2(\mathbb{C})$ and $(\omega_k)_{k=1}^{\infty}$ is a sequence of normal faithful states on $M_2(\mathbb{C})$. Let $\varphi$ be a fixed normal faithful state on $N$. Since $N$ is properly infinite, we have $N\otimes M_2(\mathbb{C})\cong N$. 
Moreover, by Connes-St\o rmer transitivity theorem \cite{ConnesStormer78}, we can choose a $*$-isomorphism $\Phi\colon N\otimes M_2(\mathbb{C})\to N$ such that 
\begin{equation*}
\|(\varphi\otimes \omega_1)\circ \Phi^{-1}-\varphi\|<\frac{1}{2}. 
\end{equation*}
Put $F_1=\Phi (\mathbb{C}\otimes M_2(\mathbb{C})),$ and $\varphi_1=(\varphi\otimes \omega_1)\circ \Phi^{-1}$. Then 
$F_1$ is a type I$_2$ subfactor of $N$. Moreover, it holds that $N\cong F_1\otimes F_1^{\text{c}}$, where $F_1^{\text{c}}=F_1'\cap N$ is the relative commutant, and $\varphi_1=\varphi_1|_{F_1}\otimes \varphi_1|_{F_1^{\text{c}}}$. Moreover, we have
\begin{equation*}
(F_1,\varphi_1|_{F_1})\cong (P_1,\omega_1). 
\end{equation*}
Using the same arguments to the type III$_1$ factor $F_1^{\text{c}}$, we can find a type I$_2$-subfactor $F_2\subset F_1^{\text{c}}$, a normal faithful state $\varphi_2'$ on $F_1^{\text{c}}$, such that $\|\varphi_2'-\varphi_1|_{F_1^{\text{c}}}\|<\frac{1}{4}$,  
\begin{equation*}
\varphi_2'=\varphi_2'|_{F_2}\otimes \varphi_2'|_{(F_1\otimes F_2)^{\text{c}}}, 
\end{equation*}
and $(F_2,\varphi_2'|_{F_2})\cong (P_2,\omega_2)$. Thus, if we put $\varphi_2=\varphi_1|_{F_1}\otimes \varphi_2'$, we have 
$\|\varphi_1-\varphi_2\|<\frac{1}{4}$,
\begin{equation*}
\varphi_2=\varphi_2|_{F_1}\otimes \varphi_2|_{F_2}\otimes \varphi_2|_{(F_1\otimes F_2)^{\text{c}}},
\end{equation*}
and 
\begin{align*}
(F_1,\varphi_2|_{F_1})\cong (P_1,\omega_1),\\
(F_2,\varphi_2|_{F_2})\cong (P_2,\omega_2).
\end{align*}
Proceeding in this way, we obtain a sequence $(F_k)_{k=1}^{\infty}$ of mutually commuting type I$_2$-subfactors of $N$, and a sequence $(\varphi_k)_{k=1}^{\infty}$ of normal faithful states on $N$, such that 
\begin{equation*}
\|\varphi_k-\varphi_{k-1}\|<2^{-k},\ \ \ \ \ k\ge 2,
\end{equation*}
and 
such that for fixed $m\in \mathbb{N}$: 
\begin{equation*}
\varphi_m=\varphi_m|_{F_1}\otimes \varphi_m|_{F_2}\otimes \ldots \otimes \varphi_m|_{F_m}\otimes \varphi_m|_{(F_1\otimes \ldots \otimes F_m)^{\text{c}}},
\end{equation*}
and 
\begin{equation*}
(F_i,\varphi_m|_{F_i})\cong (P_i,\omega_i)\ \ \ \  \ i=1,\ldots, m. 
\end{equation*}
Let $\varphi$ be the norm limit in $N_*$ of the sequence $(\varphi_k)_{k=1}^{\infty}$. 
Then $\varphi$ is a normal state, but it can fail to be faithful. From the properties of $\varphi_k$, we have for all $m\in \mathbb{N}$, 
\begin{equation*}
\varphi=\varphi|_{F_1}\otimes \varphi|_{F_2}\otimes \ldots \otimes \varphi|_{F_m}\otimes \varphi_m|_{(F_1\otimes \ldots \otimes F_m)^{\text{c}}},
\end{equation*}
and 
\begin{equation*}
(F_m,\varphi|_{F_m})\cong (P_m,\omega_m).
\end{equation*}
Let $r_k$ be the ratio between the largest and the smallest eigenvalues of $\text{d}\omega_k/\text{d}\text{Tr}$. Let $u\in \mathcal{U}(P_k)$.  
We may assume that $\omega_k=\text{Tr}(h_k\ \cdot\ ),\, h_k:=\frac{1}{1+r_k}\text{diag}(r_k,1)\ (r_k\ge 1)$. Then if $a=\mattwo{x}{y}{z}{w}\in P_k$ is positive, then 
\begin{align*}
u\omega_ku^*(a)&=\text{Tr}(h_ku^*au)=\text{Tr}((u^*au)^{\frac{1}{2}}h_k(u^*au)^{\frac{1}{2}})\notag \\
&\ge \frac{1}{1+r_k}\text{Tr}(u^*au)=\frac{x+w}{1+r_k}\notag \\
&\ge r_k^{-1}(\frac{r_kx}{1+r_k}+\frac{w}{1+r_k})\notag \\
&=r_k^{-1}\omega_k(a).
\end{align*}
Similarly, $u\omega_ku^*(a)\le r_k\omega_k(a)$ holds. This shows that 

\begin{equation*}
r_k^{-1}\omega_k \le u\omega_k u^*\le r_k\omega_k.
\end{equation*}
Hence for all $u\in \mathcal{U}(F_k)$, 
\begin{equation*}
r_k^{-1}\varphi \le u\varphi u^*\le r_k\varphi.
\end{equation*}
Thus, $\varphi$ and $u\varphi u^*$ have the same support projection in $N$, i.e., with $e=\text{supp}(\varphi)$, we have 
\begin{equation*}
ueu^*=e,\ \ \ \ u\in \mathcal{U}(F_k),\ \ k\in \mathbb{N}.
\end{equation*}
This shows that $e\in \left (\bigcup_{k=1}^{\infty}F_k\right )'\cap N$. Put $G_k=eF_k$. Then $(G_k)_{k=1}^{\infty}$ is a sequence of commuting subfactors of $eNe$. Moreover, 
the restriction $\varphi_e$ of $\varphi$ to $eNe$ is a normal faithful state on $eNe$, and 
\begin{equation*}
(G_1\otimes \ldots \otimes G_m,\varphi_e|_{G_1\otimes \ldots \otimes G_m})\cong \bigotimes_{k=1}^m(P_k,\omega_k)
\end{equation*}
for all $m\in \mathbb{N}$. Let $P$ be the von Neumann algebra generated by $\bigcup_{k=1}^{\infty}G_k$. Then 
\begin{equation*}
(P,\varphi|_P)\cong \bigotimes_{k=1}^{\infty}(P_k,\omega_k).
\end{equation*}
In particular, $P\cong R_{\infty}$. Moreover, since 
\begin{equation*}
\varphi_e=\varphi_e|_{G_1}\otimes \ldots \varphi_e|_{G_m}\otimes \varphi_e|_{(G_1\otimes \ldots \otimes G_m)^{\text{c}}},
\end{equation*}
where $(G_1\otimes \dots \otimes G_m)^{\text{c}}$ denotes the relative commutant of $\bigcup_{k=1}^mG_k$ in $eNe$, we have 
\begin{equation*}
\sigma_t^{\varphi_e}(G_1\otimes \dots \otimes G_m)=G_1\otimes \dots \otimes G_m,\ \ \ \ t\in \mathbb{R}
\end{equation*}
for all $m\in \mathbb{N}$, and hence also $\sigma_t^{\varphi_e}(P)=P,\ \ t\in \mathbb{R}$. 
Thus by \cite{Takesaki72}, there exists a normal faithful conditional expectation of $eNe$ onto $P$. This completes the proof, since $eNe$ is isomorphic to $N$. 
\end{proof}
\begin{proof}[Proof of Theorem \ref{thm: UH4.2}]
Let $N$ be a type III$_1$ factor with separable predual. By Lemmata \ref{lem: UH4.3} and \ref{lem: UH4.4}, we can choose a normal faithful conditional expectation $E$ of $N$ onto a subfactor $P$ of $N$ isomorphic to $R_{\infty}$. Moreover, we can choose a $\mathbb{Q}$-stable normal faithful state $\omega$ on $P$. Put $\varphi=\omega\circ E$. Then it follows from the bimodule property of conditional expectations \cite[Theorem 1]{Tomiyama58} that $\varphi$ is a $\mathbb{Q}$-stable normal faithful state on $N$. 
\end{proof}
\begin{theorem}\label{thm: UH4.5}
Let $\varphi$ be a $\mathbb{Q}$-stable normal faithful state on a von Neumann algebra $N$, let $m\in \mathbb{N}$, and let $q_1,\dots,q_m$ be $m$ positive rational numbers with sum 1. Then there exists a type {\rm{I}}$_m$ subfactor $F$ of $N$, such that
\begin{list}{}{}
\item[{\rm{(a)}}] $\varphi=\varphi|_{F}\otimes \varphi|_{F^{{\rm{c}}}}$. 
\item[{\rm{(b)}}] $\varphi|_{F^{{\rm{c}}}}$ is $\mathbb{Q}$-stable.
\item[{\rm{(c)}}] ${\rm{d}}\varphi|_{F}/{\rm{d}}{\rm{Tr}}_F$ has eigenvalues $(q_1,\dots,q_m)$. 
\end{list}  
Here, ${\rm{Tr}}_F$ denotes the trace on $F$ for which ${\rm{Tr}}_F(1)=m$. 
\end{theorem}
We prove first:  
\begin{lemma}\label{lem: UH4.6}
Let $\varphi$ be a $\mathbb{Q}$-stable normal faithful state on a von Neumann algebra $N$, and let $q_1,\dots,q_m$ be positive rational numbers with sum 1. Then there exist isometries $u_1,\dots,u_m\in N$ with orthogonal ranges, such that 
\begin{align*}
\sum_{i=1}^m&u_iu_i^*=1,\\
\varphi u_i&=q_iu_i\varphi,\ \ \ \ \ \ \ i=1,\dots,m.
\end{align*}
\end{lemma}
\begin{proof}
We can choose integers $p,p_1,\dots,p_m\in \mathbb{N}$ such that 
\begin{equation*}
q_i=\frac{p_i}{p},\ \ \ \ \ i=1,\dots,m.
\end{equation*}
Note that $\sum_{i=1}^mp_i=p$. By Definition \ref{def: UH4.1}, for each $i\in \{1,\dots,m\}$ we can choose $p_i$ isometries $v_{i1},\dots, v_{ip_i}$ in $N$ with orthogonal ranges, such that  
\begin{equation*}
\sum_{j=1}^{p_i}v_{ij}v_{ij}^*=1\ \ \ \text{\ and\ }\ \ \varphi v_{ij}=\frac{1}{p_i}v_{ij}\varphi,\ \ \ \ j=1,\dots,p_i.
\end{equation*}
Moreover, since the set $\{(i,j);\ 1\le i\le m,\ 1\le j\le p_i\}$ contains $\sum_{i=1}^mp_i=p$ elements, we can also find isometries $w_{ij}\in N\ (1\le i\le m,\ 1\le j\le p_i)$ with orthogonal ranges, such that 
\begin{equation*}
\sum_{i=1}^m\sum_{j=1}^{p_i}w_{ij}w_{ij}^*=1\ \ \ \text{\ and\ }\ \ \ \varphi w_{ij}=\frac{1}{p}w_{ij}\varphi,\ \ 1\le i\le m,\ 1\le j\le p_i. 
\end{equation*}
Put now 
\begin{equation*}
u_i:=\sum_{j=1}^{p_i}w_{ij}v_{ij}^*,\ \ \ \ \ i=1,\dots,m. 
\end{equation*}
Then 
\begin{align*}
u_i^*u_i&=\sum_{j=1}^{p_i}v_{ij}v_{ij}^*=1,\\
\sum_{i=1}^{m}u_iu_i^*&=\sum_{i=1}^m\sum_{j=1}^{p_i}w_{ij}w_{ij}^*=1,
\end{align*}
and since $\varphi w_{ij}=\frac{1}{p}w_{ij}\varphi$ and $v_{ij}^*\varphi =\frac{1}{p_i}\varphi v_{ij}^*$ for all $(i,j)$, we get 
\begin{equation*}
\varphi u_i=\sum_{j=1}^{p_i}\varphi w_{ij}v_{ij}^*=\sum_{j=1}^{p_i}\frac{p_i}{p}w_{ij}v_{ij}^*\varphi=q_iu_i\varphi.
\end{equation*}
This proves Lemma \ref{lem: UH4.6}. 
\end{proof}
\begin{proof}[Proof of Theorem \ref{thm: UH4.5}] 
Choose $m$ isometries $u_1,\dots,u_m\in N$ satisfying the conditions in Lemma \ref{lem: UH4.6}. We can define a $*$-isomorphism $\Phi$ of $N\otimes M_m(\mathbb{C})$ onto $N$ by 
\begin{equation*}
\Phi\left (\sum_{i,j=1}^mx_{ij}\otimes e_{ij}\right ):=\sum_{i,j=1}^mu_ix_{ij}u_j^*,
\end{equation*}
where $(e_{ij})_{i,j=1}^m$ are the matrix units in $M_m(\mathbb{C})$. Then using $\varphi u_i=\lambda_i u_i \varphi$, we get 
\begin{align*}
(\varphi\circ \Phi)\left (\sum_{i,j=1}^mx_{ij}\otimes e_{ij}\right )&=\sum_{i,j=1}^m\varphi(u_ix_{ij}u_j^*)\notag \\
&=\sum_{i,j=1}^mq_i\varphi(x_{ij}u_j^*u_i)\notag \\
&=\sum_{i=1}^mq_i\varphi(x_i).
\end{align*}
Hence 
\begin{equation*}
\varphi \circ \Phi=\varphi \otimes \omega,
\end{equation*}
where $\omega$ is the state on $M_m(\mathbb{C})$ for which 
\begin{equation*}
\frac{\text{d}\omega}{\text{d}\text{Tr}}=\begin{pmatrix}
q_1 &0 & \cdots & 0\\
0  & q_2 & &\vdots \\
\vdots & & \ddots & 0\\
 0 &\cdots & 0 & q_m
\end{pmatrix}.
\end{equation*}
Put now $F:=\Phi(\mathbb{C}\otimes M_m(\mathbb{C}))$. Then the relative commutant of $F$ in $N\otimes M_m(\mathbb{C})$ is $\Phi(N\otimes \mathbb{C})$. Since $\varphi\circ \Phi=\varphi\otimes \omega$, $\varphi$ itself is a tensor product state with respect to the decomposition 
\begin{equation*}
N=F\cdot F^{\text{c}}\cong F\otimes F^{\text{c}}.
\end{equation*}
Moreover, $\text{d}\varphi|_F/\text{dTr}_F$ has eigenvalues $(q_1,\dots,q_m)$. Let $\Phi_0$ be the isomorphism of $N$ onto $F^{\text{c}}$ given by
\begin{equation*}
\Phi_0(x)=\Phi(x\otimes 1),\ \ \ \ \ x\in N.
\end{equation*} 
Then $\varphi|_F\circ \Phi_0=\varphi$. Therefore $\varphi|_{F^{\text{c}}}$ is a $\mathbb{Q}$-stable normal faithful state on $F^{\text{c}}$. 
\end{proof}

\section{Proof of Main Theorem}
In this section we prove the main result of the paper: 
\begin{theorem}\label{thm: UH5.6}
Every injective factor $N$ of type {\rm{III}}$_1$ on a separable Hilbert space is isomorphic to the Araki--Woods factor $R_{\infty}$. 
\end{theorem}
We need preparations. In this section, for each von Neumann algebra $N$, we fix a standard form $(N,H,J,\mathcal{P}^{\natural})$. For each $\varphi\in (N_*)_+$, we denote by $\xi_{\varphi}$ the unique representing vector in $\mathcal{P}^{\natural}$ \cite{Haagerup75}.  
\begin{lemma}\label{lem: UH5.1}
Let $N$ be a properly infinite factor with separable predual and with a normal faithful state $\varphi$, let $F$ be a finite dimensional $\sigma^{\varphi}$-invariant subfactor of $N$, and let $T\colon F\to N$ be a unital completely positive map, such that $\varphi\circ T=\varphi$ and 
\begin{equation}
\|\sigma_t^{\varphi}\circ T-T\circ \sigma_t^{\varphi|_F}\|\le \delta |t|,\ \ \ \ \ \ t\in \mathbb{R},\label{eq: 5.1condition on T}
\end{equation}
where $\delta>0$ is a constant. Then there exists a norm-continuous map $a\colon \mathbb{R}\to N$ such that 
\begin{list}{}{}
\item[{\rm{(a)}}] $\displaystyle \int_{-\infty}^{\infty}a(t)^*a(t)\,dt=1$ ($\sigma$-strongly), 
\item[{\rm{(b)}}] $\displaystyle \int_{-\infty}^{\infty}e^{-t}\varepsilon_{F,\varphi}(a(t)a(t)^*)\,dt=1$ ($\sigma$-strongly), 
\item[{\rm{(c)}}] $\displaystyle \int_{-\infty}^{\infty}\|a(t)\xi_{\varphi}-e^{-t/2}\xi_{\varphi}a(t)\|^2\,dt<\frac{\delta}{8}$,
\item[{\rm{(d)}}] $\displaystyle \left \|T(x)-\int_{-\infty}^{\infty}a(t)^*xa(t)\,dt\right \|\le \delta^{\frac{1}{2}}\|x\|,\ \ \ \ x\in F$, 
\end{list}
where $\varepsilon_{F,\varphi}$ is the normal faithful conditional expectation of $N$ onto $F$ that leaves the state $\varphi$ invariant. 
\end{lemma}
\begin{proof}
Let $f$ be the function 
\begin{equation*}
f(t):=(\pi \delta)^{-\frac{1}{4}}\exp \left (-\frac{1}{2\delta}t^2\right ),\ \ \ \ \ t\in \mathbb{R},
\end{equation*}
and let $g$ be the Fourier-transformed of $f$:
\begin{equation*}
g(s):=\left (\frac{\delta}{\pi}\right)^{\frac{1}{4}}\exp \left (-\frac{\delta}{2}s^2\right ),\ \ \ \ \ s\in \mathbb{R}.
\end{equation*}
Note that 
\begin{equation*}
\int_{-\infty}^{\infty}f(t)^2dt=\int_{-\infty}^{\infty}g(s)^2\,ds=1.
\end{equation*}
By \cite[Proposition 2.1]{Haagerup85}, there exists an operator $a\in N$ such that 
\begin{equation*}
T(x)=a^*xa,\ \ \ \ \ x\in F. 
\end{equation*}
In particular, $a^*a=1$, i.e., $a$ is an isometry. Put
\begin{equation*}
a(t)=\frac{1}{\sqrt{2\pi}}\int_{-\infty}^{\infty}e^{-is(t-\delta/4)}g(s)\sigma_s^{\varphi}(a)\,ds,\ \ \ \ t\in \mathbb{R}.
\end{equation*}
Since $t\mapsto e^{-is(t-\delta/4)}g(s)$ is a continuous map from $\mathbb{R}$ to $L^1(\mathbb{R})$, the map $t\mapsto a(t)$ is a norm-continuous map from $\mathbb{R}$ to $N$. Using the Plancherel formula in $L^2(\mathbb{R},H)$, we get 
\begin{equation*}
\int_{-\infty}^{\infty}\|a(t)\xi\|^2\,dt=\int_{-\infty}^{\infty}g(s)^2\|\sigma_s^{\varphi}(a)\xi\|^2\,ds=\|\xi\|^2
\end{equation*}
for all $\xi\in H$. Hence 
\begin{equation*}
\int_{-\infty}^{\infty}a(t)^*a(t)\,dt=1\ \ \ \ (\sigma\text{-weakly}).
\end{equation*}
Since the convergence of the integral is monotone, we get (a). Using again the Plancherel formula, we get for $\xi,\eta \in H$ and $x\in F$, 
\begin{align*}
\int_{-\infty}^{\infty}\nai{xa(t)\xi}{a(t)\eta}\,dt&=\int_{-\infty}^{\infty}g(s)^2\nai{x\sigma_s^{\varphi}(a)\xi}{\sigma_s^{\varphi}(a)\eta}\,ds\notag \\
&=\int_{-\infty}^{\infty}g(s)^2\nai{\sigma_s^{\varphi}\circ T\circ \sigma_{-s}^{\varphi}(x)\xi}{\eta}\,ds.
\end{align*}
Hence for $x\in F$,
\begin{equation}
\int_{-\infty}^{\infty}a(t)^*xa(t)\,dt=\int_{-\infty}^{\infty}g(s)^2\sigma_s^{\varphi}\circ T\circ \sigma_{-s}^{\varphi}(x)\,ds.\label{eq: integral of a(t)*xa(t)}
\end{equation}
Note that the left hand side of (\ref{eq: integral of a(t)*xa(t)}) converges $\sigma$-strongly, because $F=\text{span}(F_+)$ and for $x\in F_+$, the integral converges $\sigma$-weakly and the convergence is monotone.   
Therefore by (\ref{eq: 5.1condition on T}), for each $x\in F$ we get
\begin{align*}
\left \|T(x)-\int_{-\infty}^{\infty}a(t)^*xa(t)dt\right \|&\le \delta \|x\|\int_{-\infty}^{\infty}|s|g(s)^2\,ds\notag \\
&=\left (\frac{\delta}{\pi}\right )^{\frac{1}{2}}\|x\|\notag\\
&\le \delta^{\frac{1}{2}}\|x\|.
\end{align*}
This proves (d). Since $g(s)$ has the analytic extension to the function $g\colon \mathbb{C}\to \mathbb{C}$, and since the integrals
\begin{equation*}
\int_{-\infty}^{\infty}|g(s+iu)|\,ds=\left(\frac{4\pi}{\delta}\right )^{\frac{1}{4}}e^{\frac{\delta}{2}u^2},\ \ \ \ u\in \mathbb{R}
\end{equation*}
are uniformly bounded for $u$ on bounded subsets of $\mathbb{R}$, it follows that $a(t)$ is analytic with respect to $\sigma^{\varphi}$ (in the sense of \cite{PedersenTakesaki73}) and that 
\begin{equation*}
\sigma_{\alpha}^{\varphi}(a(t))=\frac{1}{\sqrt{2\pi}}\int_{-\infty}^{\infty}e^{-i(s-\alpha)(t-\tfrac{\delta}{4})}g(s-\alpha)\sigma_s^{\varphi}(a)\,ds,
\end{equation*}
for all $\alpha\in \mathbb{C}$. To prove (c), we use the equality
\begin{equation*}
\xi_{\varphi}a(t)=J_{\varphi}a(t)^*\xi_{\varphi}=\Delta_{\varphi}^{\frac{1}{2}}a(t)\xi_{\varphi}=\sigma_{-i/2}^{\varphi}(a(t))\xi_{\varphi}.
\end{equation*}
Hence 
\begin{equation*}
e^{-t/2}\xi_{\varphi}a(t)=\frac{e^{-\tfrac{\delta}{8}}}{\sqrt{2\pi}}\int_{-\infty}^{\infty}e^{-is(t-\tfrac{\delta}{4})}g(s+\tfrac{i}{2})\sigma_s^{\varphi}(a)\xi_{\varphi}\,ds.
\end{equation*}
Using the Plancherel formula, we get 
\begin{equation*}
\int_{-\infty}^{\infty}\|a(t)\xi_{\varphi}-e^{-t/2}\xi_{\varphi}a(t)\|^2\,dt=\int_{-\infty}^{\infty}|g(s)-e^{-\tfrac{\delta}{8}}g(s+\tfrac{i}{2})|^2\|a\xi_{\varphi}\|^2\,ds.
\end{equation*}
On the other hand, $g(s+\tfrac{i}{2})$ is the Fourier--Plancherel transformed of $e^{t/2}f(t)$. Therefore the above integral is equal to 
\begin{equation*}
\int_{-\infty}^{\infty}f(t)^2\left (1-e^{-\tfrac{\delta}{8}+\tfrac{t}{2}}\right )^2\,dt.
\end{equation*}
It is easy to compute that for $\gamma \in \mathbb{R}$, 
\begin{equation*}
\int_{-\infty}^{\infty}f(t)^2e^{\gamma t}\,dt=\exp (\tfrac{1}{4}\gamma^2\delta).
\end{equation*}
Therefore 
\begin{align*}
\int_{-\infty}^{\infty}f(t)^2\left (1-e^{-\tfrac{\delta}{8}+\tfrac{t}{2}}\right )^2\,dt&=2(1-e^{-\tfrac{\delta}{16}})\notag \\
&<\frac{\delta}{8}.
\end{align*}
This proves (c). Put now 
\begin{equation*}
A(t):=e^{-t}\varepsilon_{F,\varphi}(a(t)a(t)^*),\ \ \ \ \ t\in \mathbb{R}.
\end{equation*}
Since $\varepsilon_{F,\varphi}$ is a normal faithful conditional expectation of $N$ onto $F$ that leaves $\varphi$ invariant, we have for $x\in F$, that 
\begin{equation*}
\varphi(A(t)x)=e^{-t}\varphi(a(t)a(t)^*x).
\end{equation*}
By the KMS-condition, it follows that if $a,b\in N$ and $a$ is $\sigma^{\varphi}$-analytic, then 
\begin{equation*}
\varphi(ab)=\varphi(b\sigma_{-i}^{\varphi}(a))
\end{equation*} 
(cf. \cite[Theorem 3.2]{Haagerup79}). 
Hence for $x\in F$, 
\begin{equation*}
\varphi(A(t)x)=e^{-t}\varphi(a(t)^*x\sigma_{-i}^{\varphi}(a(t))). 
\end{equation*}
Using 
\begin{equation*}
e^{-t}\sigma_{-i}^{\varphi}(a(t))=\frac{e^{-\tfrac{\delta}{4}}}{\sqrt{2\pi}}\int_{-\infty}^{\infty}e^{-is(t-\tfrac{\delta}{4})}g(s+i)\sigma_s^{\varphi}(a)\,ds,
\end{equation*}
we get by the Plancherel formula, that 
\begin{align*}
\int_{-\infty}^{\infty}\varphi(A(t)x)\,dt&=\int_{-\infty}^{\infty}\nai{x(e^{-t}\sigma_{-i}^{\varphi}(a(t)))\xi_{\varphi}}{a(t)\xi_{\varphi}}\,dt\notag \\
&=e^{-\tfrac{\delta}{4}}\int_{-\infty}^{\infty}g(s+i)\overline{g(s)}\nai{x\sigma_s^{\varphi}(a)\xi_{\varphi}}{\sigma_s^{\varphi}(a)\xi_{\varphi}}\,ds.
\end{align*}
Since $\varphi \circ T=\varphi$, it holds that 
\begin{equation*}
\nai{x\sigma_s^{\varphi}(a)\xi_{\varphi}}{\sigma_s^{\varphi}(a)\xi_{\varphi}}=\varphi \circ \sigma_{-s}^{\varphi}\circ T\circ \sigma_{s}^{\varphi}(x)=\varphi(x),
\end{equation*}
Hence 
\begin{equation*}
\int_{-\infty}^{\infty}\varphi(A(t)x)\,dt=e^{-\tfrac{\delta}{4}}\varphi(x)\int_{-\infty}^{\infty}g(s+i)\overline{g(s)}\,ds.
\end{equation*}
Since $g(s+i)$ is the Fourier--Plancherel transformed of $f(t)e^t$, we get
\begin{equation*}
\int_{-\infty}^{\infty}g(s+i)\overline{g(s)}\,ds=\int_{-\infty}^{\infty}|f(t)|^2e^t\,dt=e^{\frac{\delta}{4}}.
\end{equation*}
Since $F$ is finite-dimensional and $\varphi$ is faithful on $F$, every $\psi\in F_*$ is of the form $\varphi(\ \cdot\ x),\ x\in F$. 
This shows that 
\begin{equation*}
\int_{-\infty}^{\infty}\psi(A(t))\,dt=\psi(1),\ \ \ \ \psi \in F_*,
\end{equation*}
that is, we have
\begin{equation*}
\int_{-\infty}^{\infty}A(t)\,dt=1\ \ \ \ (\sigma \text{-weakly}).
\end{equation*}
This proves (b). 
\end{proof}
\begin{lemma}\label{lem: UH5.2}
Let $N,\varphi,F$ and $\varepsilon_{F,\varphi}$ be as in Lemma \ref{lem: UH5.1}. Let $\lambda>0$ and assume that $c_1,\dots,c_s$ are operators in $F^{{\rm{c}}}=F'\cap N$ such that 
\begin{align*}
\varphi c_i&=\lambda c_i\varphi,\ \ \ \ \ \ i=1,\dots,s,\\
\sum_{i=1}^s&c_i^*c_i=1.
\end{align*}
Then for all $x\in N$, 
\begin{equation*}
\varepsilon_{F,\varphi}\left (\sum_{i=1}^sc_ixc_i^*\right )=\lambda \varepsilon_{F,\varphi}(x).
\end{equation*}
\end{lemma}
\begin{proof}
It is sufficient to check the formula for $x\in N$ of the form $x=ab,\ a\in F,\ b\in F^{\text{c}}$. For $z\in F^{\text{c}}$, $\varepsilon_{F,\varphi}(z)$ commutes with every element in $F$. Hence $\varepsilon_{F,\varphi}(z)$ is a scalar multiple of the identity. Using that $\varepsilon_{F,\varphi}$ leaves $\varphi$ invariant, we get
\begin{equation*}
\varepsilon_{F,\varphi}(z)=\varphi(z)1,\ \ \ \ z\in F^{\text{c}}.
\end{equation*}
Therefore 
\begin{align*}
\varepsilon_{F,\varphi}\left (\sum_{i=1}^sc_ixc_i^*\right )&=\varepsilon_{F,\varphi}\left (a\left (\sum_{i=1}^sc_ibc_i^*\right )\right )\notag \\
&=\varphi \left (\sum_{i=1}^sc_ibc_i^*\right )a\notag \\
&=\lambda \varphi \left (\sum_{i=1}^sbc_i^*c_i\right )a\notag \\
&=\lambda \varphi (b)a\notag \\
&=\lambda \varepsilon_{F,\varphi}(x).
\end{align*}
\end{proof}
\begin{lemma}\label{lem: UH5.3}
Let $\varphi$ be a $\mathbb{Q}$-stable normal faithful state on an injective factor $N$ of type {\rm{III}}$_1$ with separable predual. Let $u_1,\dots,u_n\in \mathcal{U}(N)$, let $\delta>0$. Then there exist a finite dimensional $\sigma^{\varphi}$-invariant subfactor $F$ of $N$ and unitaries $v_1,\dots,v_n\in \mathcal{U}(F)$, such that for every $\sigma$-strong neighborhood $\mathcal{V}$ of $0$ in $N$, there exists a finite set $b_1,\dots,b_r$ of operators in $N$ with the following properties: 
\begin{list}{}{}
\item[{\rm{(a)}}] $\displaystyle \sum_{i=1}^rb_i^*b_i\in 1+\mathcal{V}$\ \ \ and\ \ \ $\displaystyle \sum_{i=1}^rb_i^*b_i\le 1$.
\item[{\rm{(b)}}] $\displaystyle \varepsilon_{F,\varphi}\left (\sum_{i=1}^rb_ib_i^*\right )\in 1+\mathcal{V}$\ \ \  and\ \ \ $\displaystyle \varepsilon_{F,\varphi}\left (\sum_{i=1}^rb_ib_i^*\right )\le 1$. 
\item[{\rm{(c)}}] $\displaystyle \sum_{i=1}^r\|b_i\xi_{\varphi}-\xi_{\varphi}b_i\|^2<\delta.$
\item[{\rm{(d)}}] $\displaystyle \sum_{i=1}^r\|b_iu_k-v_kb_i\|_{\varphi}^2<\delta,\ \ \ k=1,\dots,n.$
\end{list}
\end{lemma}
\begin{proof}
Put $\delta_1=\text{min}(\delta^2/4,\delta)$. 
By Theorem \ref{thm: UH3.1}, there exist $m\in \mathbb{N}$, a unital completely positive map $T_0\colon M_m(\mathbb{C})\to N$ and unitaries $w_1,\dots,w_n\in M_m(\mathbb{C})$ such that $\psi:=\varphi\circ T\in M_m(\mathbb{C})_*$ satisfies 
\begin{align*}
\|\sigma_t^{\varphi}\circ T_0-T_0\circ \sigma_t^{\psi}\|&\le \frac{\delta_1}{2}|t|,\ \ \ t\in \mathbb{R},\\
\|T_0(w_k)-u_k\|_{\varphi}&<\varepsilon,\ \ \ k=1,\dots,n.
\end{align*}
Let $\{q_1',\dots,q_m'\}$ be the spectrum of $\displaystyle \text{d}\psi/\text{dTr}\in M_m(\mathbb{C})_+$ where the multiplicity is taken into account. 
Let $\{q_1,\dots, q_m\}$ be positive rationals with sum 1, and let  $\chi$ on $M_m(\mathbb{C})_+$ such that $\text{d}\chi/\text{dTr}$ has the same spectral projections as $\text{d}\psi/\text{dTr}$ but the eigenvalues replaced by $\{q_1,\dots, q_m\}$.  Since $\|e^{ia}-e^{ib}\|\le \|a-b\|$ for self-adjoint operators $a,b$, (cf. (\ref{eq: hit-h_0it}) in Theorem \ref{thm: UH3.1}), we may arrange $q_i$'s so that the following inequality holds:  
\begin{equation*}
\|\sigma_t^{\psi}-\sigma_t^{\chi}\|_{M_m(\mathbb{C})}\le \frac{\delta_1}{2}|t|,\ \ \ \ \ t\in \mathbb{R}.
\end{equation*}
Since $\varphi$ is $\mathbb{Q}$-stable and $q_i$'s are rationals, by Theorem \ref{thm: UH4.5}, there exists a finite-dimensional subfactor $F\subset N$ and a state-preserving $*$-isomorphism $\Phi\colon (M_m(\mathbb{C}),\chi)\to (F,\varphi|_F)$ such that $\varphi|_{F^{\text{c}}}$ is $\mathbb{Q}$-stable. Define $T:=T_0\circ \Phi^{-1}\colon F\to N$ and $v_k:=\Phi_0(w_k)\in \mathcal{U}(F) (1\le k\le n)$. 
Then if $x=\Phi(y)\ (y\in M_m(\mathbb{C}))$ and $t\in \mathbb{R}$, we have 
\begin{align*}
\|\sigma_t^{\varphi}\circ T(x)-T\circ \sigma_t^{\varphi|_F}(x)\|&=\|\sigma_t^{\varphi}\circ T_0(y)-T_0\circ \Phi^{-1}\circ \sigma_t^{\varphi|_F}\circ \Phi(y)\|\notag \\
&=\|\sigma_t^{\varphi}\circ T_0(y)-T_0\circ \sigma_t^{\chi}(y)\|\notag \\
&\le \|\sigma_t^{\varphi}\circ T_0(y)-T_0\circ \sigma_t^{\psi}(y)\|+\|T(\sigma_t^{\psi}(y)-\sigma_t^{\chi}(y))\|\notag \\
&\le \delta_1|t|\|y\|.
\end{align*}
Therefore we obtain
\begin{align*}
T(1)=1,\ \ \ \ &\varphi\circ T=\varphi|_F,\\
\|\sigma_t^{\varphi}\circ T-T\circ \sigma_t^{\varphi|_F}\|&\le \delta_1|t|,\ \  \ \ \ \ \ \ \ t\in \mathbb{R},\\
\|T(v_k)-u_k\|_{\varphi}&<\delta_1^{\frac{1}{2}},\ \ \ \ \ \ \  \ k=1,\dots,n.
\end{align*} 
Choose now a norm-continuous function $t\mapsto a(t)$ of $\mathbb{R}$ into $N$, such that the conditions (a),\,(b),\,(c) and (d) in Lemma \ref{lem: UH5.1} are satisfied with respect to $\delta_1$ instead of $\delta$. Then using (d), we have 
\begin{equation*}
\|u_k-\int_{-\infty}^{\infty}a(t)^*v_ka(t)\,dt\|_{\varphi}<2\delta_1^{\frac{1}{2}}\le \delta
\end{equation*}
for $k=1,\dots,n$. Using that $\displaystyle \int_{-\infty}^{\infty}a(t)^*a(t)\,dt=1$, it follows that 
\eqa{
\int_{-\infty}^{\infty}\|a(t)u_k-v_ka(t)\|_{\varphi}^2\,dt&=2-2\text{Re}\int_{-\infty}^{\infty}\nai{a(t)u_k\xi_{\varphi}}{v_ka(t)\xi_{\varphi}}\,dt\\
&=2-2\text{Re}\left (\nai{u_k\xi_{\varphi}}{\int_{-\infty}^{\infty}a(t)^*v_ka(t)\xi_{\varphi}\,dt}\right )\\
&\le \|u_k-\int_{-\infty}^{\infty}a(t)^*v_ka(t)\,dt\|_{\varphi}\\
&<\delta.
}
Let now $\mathcal{V}$ be a $\sigma$-strong neighborhood of 0 in $N$. It is no loss of generality to assume that $\mathcal{V}$ is open. For sufficiently large $\gamma\in \mathbb{R}_+$, we have:
\begin{itemize}
\item[{\rm{(a')}}] $\displaystyle \int_{-\gamma}^{\gamma}a(t)^*a(t)dt\in 1+\mathcal{V}$\ \ \  and\ \ \  $\displaystyle \int_{-\gamma}^{\gamma}a(t)^*a(t)\,dt\le 1$.
\item[{\rm{(b')}}] $\displaystyle \int_{-\gamma}^{\gamma}e^{-t}\varepsilon_{F,\varphi}(a(t)a(t)^*)dt\in 1+\mathcal{V}$\ \ \  and\ \ \  $\displaystyle \int_{-\gamma}^{\gamma}e^{-t}\varepsilon_{F,\varphi}(a(t)a(t)^*)\,dt\le 1$. 
\item[{\rm{(c')}}] $\displaystyle \int_{-\gamma}^{\gamma}\|a(t)\xi_{\varphi}-e^{-t/2}\xi_{\varphi}a(t)\|^2\,dt<\frac{\delta_1}{8}\le \delta$. 
\item[{\rm{(d')}}] $\displaystyle \int_{-\gamma}^{\gamma}\|a(t)u_k-v_ka(t)\|_{\varphi}^2\,dt<\delta.$
\end{itemize}
Since $t\mapsto a(t)$ is norm-continuous, we can approximate (in norm) the above $N$-valued Riemann integrals over $[-\gamma,\gamma]$ to get the following statements: there exists an $h_0>0$ such that when $0<h<h_0$, the operators 
$$a_j=h^{-\frac{1}{2}}a(jh),\ \ \ j\in \mathbb{Z}$$
satisfy the following relations:
\begin{itemize}
\item[{\rm{(a'')}}] $\displaystyle \sum_{j=-p}^pa_j^*a_j\in 1+\mathcal{V}.$
\item[{\rm{(b'')}}] $\displaystyle \sum_{j=-p}^pe^{-jh}\varepsilon_{F,\varphi}(a_ja_j^*)\in 1+\mathcal{V}$.
\item[{\rm{(c'')}}] $\displaystyle  \sum_{j=-p}^p\|a_j\xi_{\varphi}-e^{-\frac{1}{2}jh}\xi_{\varphi}a_j\|^2<\delta$. 
\item[{\rm{(d'')}}] $\displaystyle \sum_{j=-p}^p\|a_ju_k-v_ka_j\|_{\varphi}^2dt<\delta$,
\end{itemize}
where  $p$ is the largest integer smaller than $\gamma/h_0$. 
Moreover, since the Riemann sum is norm-convergent, by multiplying a scalar $c>0$ to $a_j$'s which is sufficiently close to 1 if necessary, we may moreover assume that 
\begin{align}
\sum_{j=-p}^pa_j^*a_j&\le 1\label{eq: suma_j^*a_j}\\
\sum_{j=-p}^pe^{-jh}\varepsilon_{F,\varphi}&(a_ja_j^*)\le 1.\label{eq: suma_ja_j^*}
\end{align}
Choose now $h\in (0,h_0)$, such that $\exp(h)\in \mathbb{Q}$. This implies that the numbers $q_j=e^{-jh},\ j\in \mathbb{Z}$ are rational. Since the restriction of $\varphi$ to $F^{\text{c}}$ is $\mathbb{Q}$-stable, there exists for each $j\in \mathbb{Z}$ a finite set of operators $c_{j1},\dots,c_{js(j)}$ in $F^{\text{c}}$ such that 
\begin{equation*}
\varphi c_{ji}=e^{-jh}c_{ji}\varphi,\ \ \ \ \ \ i=1,\dots,s(j)
\end{equation*}  
and
\begin{equation*}
\sum_{i=1}^{s(j)}c_{ji}^*c_{ji}=1.
\end{equation*}
(Here we use Lemma \ref{lem: UH4.6} together with the fact that $\varphi=\varphi|_{F}\otimes \varphi|_{F^{\text{c}}}$, when $F$ is $\sigma^{\varphi}$-invariant). Put 
\begin{equation*}
b_{ji}=c_{ji}a_j,\ \ \ \ |j|\le p,\ \ 1\le i\le s(j).
\end{equation*}
Then by (\ref{eq: suma_j^*a_j}), 
\begin{list}{}{}
\item[{\rm{(a'")}}] \hspace{1.5cm}$\displaystyle \sum_{j=-p}^p\sum_{i=1}^{s(j)}b_{ji}^*b_{ji}=\sum_{j=-p}^pa_j^*a_j\in 1+\mathcal{V}$\ \  \ and\ \ \ $\displaystyle \sum_{j=-p}^p\sum_{i=1}^{s(j)}b_{ji}^*b_{ji}\le 1$,
\end{list}
and by (\ref{eq: suma_ja_j^*}) and Lemma \ref{lem: UH5.2}, 
\begin{list}{}{}
\item[{\rm{(b'")}}] \ \ \ \ \ $\displaystyle \varepsilon_{F,\varphi} \left (\sum_{j=-p}^p\sum_{i=1}^{s(j)}b_{ji}b_{ji}^*\right )=\sum_{j=-p}^pe^{-jh}\varepsilon_{F,\varphi}(a_ja_j^*)\in 1+\mathcal{V}$,\\ 
\hspace{4.5cm}and\ \ \ $\displaystyle \varepsilon_{F,\varphi} \left (\sum_{j=-p}^p\sum_{i=1}^{s(j)}b_{ji}b_{ji}^*\right )\le 1.$
\end{list}
The equality $\varphi c_{ji}=e^{-jh}c_{ji}\varphi$ implies that 
\begin{equation*}
\xi_{\varphi}c_{ji}=e^{-\frac{1}{2}jh}c_{ji}\xi_{\varphi}. 
\end{equation*}
Therefore 
\eqa{
\text{(c'")}\ \ \sum_{j=-p}^p\sum_{i=1}^{s(j)}\|b_{ji}\xi_{\varphi}-\xi_{\varphi}b_{ji}\|^2&=\sum_{j=-p}^p\sum_{i=1}^{s(j)}\|c_{ji}(a_j\xi_{\varphi}-e^{-\frac{1}{2}jh}\xi_{\varphi}a_j)\|^2\\
&=\sum_{j=-p}^p\|a_j\xi_{\varphi}-e^{-\frac{1}{2}jh}\xi_{\varphi}a_j\|^2\\
&<\varepsilon.
}
Finally, using that $v_k\in F$ and $a_{ji}\in F^{\text{c}}$, we get 
\eqa{
\text{(d'")}\ \ \ \  \  \sum_{j=-p}^p\sum_{i=1}^{s(j)}\|b_{ji}u_k-v_kb_{ji}\|_{\varphi}^2&=\sum_{j=-p}^p\sum_{i=1}^{s(j)}\|c_{ji}(a_ju_k-v_ka_j)\|_{\varphi}^2\\
&=\sum_{j=-p}^p \|a_ju_k-v_ka_j\|_{\varphi}^2\\
&<\delta.
} 
This completes the proof of Lemma \ref{lem: UH5.3}.
\end{proof}

In the proof of the following lemma, it is essential that injective type III$_1$ factors (on a separable Hilbert space) have trivial bicentralizers.  
\begin{lemma}\label{lem: UH5.4} Let $\varphi$ be a $\mathbb{Q}$-stable normal faithful state on an injective factor $N$ of type {\rm{III}}$_1$ with separable predual. Let $u_1,\dots,u_n\in \mathcal{U}(N)$, and let $\delta>0$. Then there exists a finite dimensional $\sigma^{\varphi}$-invariant subfactor $F$ of $N$ and $v_1,\dots,v_n\in \mathcal{U}(F)$, such that the for every $\sigma$-strong neighborhood $\mathcal{V}$ of 0 in $N$, there exists a finite set $a_1,\dots,a_p$ of operators in $N$ with the following properties: 
\begin{list}{}{}
\item[{\rm{(a)}}] $\displaystyle \sum_{i=1}^pa_i^*a_i\in 1+\mathcal{V}$\ \ \ and\ \ \ $\displaystyle \sum_{i=1}^pa_i^*a_i\le 1$.
\item[{\rm{(b)}}] $\displaystyle \sum_{i=1}^pa_ia_i^*\in 1+\mathcal{V}$\ \ \ and\ \ \ $\displaystyle \sum_{i=1}^pa_ia_i^*\le 1$. 
\item[{\rm{(c)}}] $\displaystyle \sum_{i=1}^p\|a_i\xi_{\varphi}-\xi_{\varphi}a_i\|^2<\delta$. 
\item[{\rm{(d)}}] $\displaystyle \sum_{i=1}^p\|a_iu_k-v_ka_i\|_{\varphi}^2<\delta,\ \ \ \ k=1,\dots, n$.  
\end{list}
\end{lemma} 
\begin{proof}
Choose an $F$ and $v_1,\dots,v_n\in \mathcal{U}(F)$ satisfying the properties of Lemma \ref{lem: UH5.3} with respect to $(u_1,\dots,u_n,\delta)$, and let $\mathcal{V}$ be a $\sigma$-strongly open neighborhood of $0$ in $N$. By Lemma \ref{lem: UH5.3}, there exists $b_1,\dots,b_r\in N$ such that 
\begin{list}{}{}
\item[{\rm{(a')}}] $\displaystyle \sum_{i=1}^rb_i^*b_i\in 1+\mathcal{V}$\ \ \ and\ \ \ $\displaystyle \sum_{i=1}^rb_i^*b_i\le 1$.
\item[{\rm{(b')}}] $\displaystyle \varepsilon_{F,\varphi}\left (\sum_{i=1}^rb_ib_i^*\right )\in 1+\mathcal{V}$\ \ \  and\ \ \  $\displaystyle \varepsilon_{F,\varphi}\left (\sum_{i=1}^rb_ib_i^*\right )\le 1$. 
\item[{\rm{(c')}}] $\displaystyle \sum_{i=1}^r\|b_i\xi_{\varphi}-\xi_{\varphi}b_i\|^2<\delta.$
\item[{\rm{(d')}}] $\displaystyle \sum_{i=1}^r\|b_iu_k-v_kb_i\|^2<\delta,\ \ \ k=1,\dots,n.$
\end{list}
Let $\delta'>0$ and $h$ denote the operator $\sum_{i=1}^rb_ib_i^*$. Since $B_{\varphi}=\mathbb{C}1$, by Proposition \ref{prop: bicentralizer condition}, we have 
\begin{equation}
\varepsilon_{F,\varphi}(h)\in \overline{\text{conv}}\{whw^*;\ w\in \mathcal{U}(F^{\text{c}}),\ \|w\xi_{\varphi}-\xi_{\varphi}w\|<\delta'\}.\label{eq: relative Scwarz}
\end{equation}
Here, the bar in (\ref{eq: relative Scwarz}) denotes the $\sigma$-strong closure. Hence there exist $w_1,\dots,w_s\in \mathcal{U}(F^{\text{c}})$, and scalars $\lambda_1,\dots,\lambda_s\in \mathbb{R}_+$, with sum $1$, such that 
\begin{equation*}
\|w_j\xi_{\varphi}-\xi_{\varphi}w_j\|<\delta',\ \ \ \ \ j=1,\dots, s
\end{equation*}
and 
\begin{equation*}
\sum_{j=1}^s\lambda_jw_jhw_j^*\in 1+\mathcal{V}.
\end{equation*}
Put 
\begin{equation*}
a_{ij}:=\lambda_j^{\frac{1}{2}}w_jb_i,\ \ \ \ \ \ \ i=1,\dots, r,\ j=1,\dots, s. 
\end{equation*}
Then 
\begin{list}{}{}
\item[{\rm{(a")}}] $\displaystyle \sum_{i=1}^r\sum_{j=1}^sa_{ij}^*a_{ij}=\sum_{i=1}^rb_i^*b_i\in 1+\mathcal{V}$\ \ \ and\ \ \ $\displaystyle \sum_{i=1}^r\sum_{j=1}^sa_{ij}^*a_{ij}\le 1$. 
\item[{\rm{(b'')}}] $\displaystyle \sum_{i=1}^r\sum_{j=1}^sa_{ij}a_{ij}^*=\sum_{j=1}^s\lambda_jw_jhw_j^*\in 1+\mathcal{V}$\ \ \ and\ \ \ $\displaystyle \sum_{i=1}^r\sum_{j=1}^sa_{ij}a_{ij}^*\le 1$.  
\end{list}
Moreover, using 
\begin{equation*}
a_{ij}\xi_{\varphi}-\xi_{\varphi}a_{ij}=\lambda_j^{\frac{1}{2}}w_j(b_i\xi_{\varphi}-\xi_{\varphi}b_i)+\lambda_j^{\frac{1}{2}}(w_j\xi_{\varphi}-\xi_{\varphi}w_j)b_i,
\end{equation*}
we obtain 
\eqa{
\text{(c")}\ \ \left (\sum_{i=1}^r\sum_{j=1}^s\|a_{ij}\xi_{\varphi}-\xi_{\varphi}a_{ij}\|^2\right )^{\frac{1}{2}}&\le \left (\sum_{i=1}^r\sum_{j=1}^s\lambda_j\|b_i\xi_{\varphi}-\xi_{\varphi}b_i\|^2\right )^{\frac{1}{2}}+\delta' \left (\sum_{i=1}^r\sum_{j=1}^s\lambda_j\|b_i\|^2\right )^{\frac{1}{2}}\\
&=\left (\sum_{i=1}^r\|b_i\xi_{\varphi}-\xi_{\varphi}b_i\|^2\right )^{\frac{1}{2}}+\delta'\left (\sum_{i=1}^r\|b_i\|^2\right )^{\frac{1}{2}}.
}
Finaly, since $v_k\in F$ and $w_j\in F^{\text{c}}$, we have 
\eqa{
\text{(d")}\ \ \ \ \sum_{i=1}^r\sum_{j=1}^s\|a_{ij}u_k-v_ka_{ij}\|_{\varphi}^2&=\sum_{i=1}^r\sum_{j=1}^s\lambda_j\|w_j(b_iu_k-v_kb_i)\|_{\varphi}^2\\
&=\sum_{i=1}^r\|b_iu_k-v_kb_i\|_{\varphi}^2\\
&<\delta.
}
Since $\delta'>0$ was arbitrary (independent of $\delta,\mathcal{V}$, and $b_1,\dots, b_r$), we can assume that 
\begin{equation*}
\left (\sum_{i=1}^r\|b_i\xi_{\varphi}-\xi_{\varphi}b_i\|^2\right )^{\frac{1}{2}}+\delta'\left (\sum_{i=1}^r\|b_i\|^2\right )^{\frac{1}{2}}<\delta^{\frac{1}{2}}.
\end{equation*}
This proves Lemma \ref{lem: UH5.4}. 
\end{proof}
\begin{lemma}\label{lem: UH5.5} 
Let $N$ be an injective factor of type {\rm{III}}$_1$ with separable predual, and let $\varphi$ be a $\mathbb{Q}$-stable normal faithful state on $N$. Let $u_1,\dots,u_n\in \mathcal{U}(N)$ and let $\varepsilon>0$. Then there exist a $\sigma^{\varphi}$-invariant finite dimensional subfactor $F$ of $N$, $v_1,\dots,v_n\in \mathcal{U}(F)$ and a unitary $w\in \mathcal{U}(N)$ such that 
\begin{equation*}
\|w\xi_{\varphi}-\xi_{\varphi}w\|<\varepsilon,
\end{equation*}
and 
\begin{equation*}
\|w^*v_kw-u_k\|_{\varphi}<\varepsilon,\ \ \ \ \ k=1,\dots, n.
\end{equation*}
\end{lemma}
\begin{proof}
Let $\delta(n,\varepsilon)>0$ be the function in Theorem \ref{thm: almost unitary equivalence}, and put $\delta_1=\frac{1}{16}\delta(n+1,\varepsilon/2)$. Choose $F$ and $v_1,\dots,v_n\in \mathcal{U}(F)$, such that the conditions of Lemma \ref{lem: UH5.4} are satisfied with respect to $(u_1,\dots,u_n,\delta_1)$. Put 
\begin{equation*}
\xi_k=u_k\xi_{\varphi},\ \ \eta_k=v_k\xi_{\varphi},\ \ \ \ k=1,\dots, n.
\end{equation*}
For every $\sigma$-strong neighborhood $\mathcal{V}$ of $0$ in $N$, there exist $a_1,\dots,a_p\in N$, such that (a), (b), (c) and (d) in Lemma \ref{lem: UH5.4} are satisfied. Since 
\begin{equation*}
a_i\xi_k-\eta_ka_i=(a_iu_k-v_ka_i)\xi_{\varphi}+v_k(a_i\xi_{\varphi}-\xi_{\varphi}a_i),
\end{equation*}
we have 
\begin{align*}
\left (\sum_{i=1}^p\|a_i\xi_k-\eta_ka_i\|^2\right )^{\frac{1}{2}}&\le \left (\sum_{i=1}^p\|a_iu_k-v_ka_i\|_{\varphi}^2\right )^{\frac{1}{2}}+\left (\sum_{i=1}^p\|a_i\xi_{\varphi}-\xi_{\varphi}a_i\|^2\right )^{\frac{1}{2}}\notag \\
&<2\delta_1^{\frac{1}{2}}.
\end{align*}
Moreover, 
\begin{equation*}
\left (\sum_{i=1}^p\|a_i\xi_{\varphi}-\xi_{\varphi}a_i\|^2\right )^{\frac{1}{2}}<\delta_1^{\frac{1}{2}}<2\delta_1^{\frac{1}{2}}.
\end{equation*}
Since $\sum_{i=1}^pa_i^*a_i\in 1+\mathcal{V}$, $\sum_{i=1}^pa_i^*a_i\le 1$, $\sum_{i=1}^pa_ia_i^*\in 1+\mathcal{V}$ and $\sum_{i=1}^pa_ia_i^*\le 1$, the two $(n+1)$-tuples $(\xi_1,\dots,\xi_n,\xi_{\varphi})$ and $(\eta_1,\dots,\eta_n,\xi_{\varphi})$ satisfies the conditions of Remark \ref{rem: approximate related} with $4\delta_1^{\frac{1}{2}}$ instead of $\delta$, so that the two $(n+1)$-tuples are $16\delta_1$-related, or equivalently they are $\delta(n+1,\frac{\varepsilon}{2})$-related in the sense of Remark \ref{rem: approximate related}. Hence by Theorem \ref{thm: almost unitary equivalence}, there exists a unitary operator $w\in \mathcal{U}(N)$, such that 
\begin{equation*}
\|w\xi_k-\eta_kw\|<\frac{\varepsilon}{2},\ \ \ \ k=1,\dots, n. 
\end{equation*}
and
\begin{equation*}
\|w\xi_{\varphi}-\xi_{\varphi}w\|<\frac{\varepsilon}{2}.
\end{equation*}
Therefore 
\begin{align*}
\|w^*v_kw-u_k\|_{\varphi}&=\|w^*(v_kw-wu_k)\xi_{\varphi}\|\notag \\
&=\|(wu_k-v_kw)\xi_{\varphi}\|\notag \\
&=\|(w\xi_k-\eta_kw)+v_k(\xi_{\varphi}w-w\xi_{\varphi})\|\notag \\
&<\varepsilon,
\end{align*}
which completes the proof of Lemma \ref{lem: UH5.5}. 
\end{proof}
Now we are ready to prove the main theorem of the paper. 
\begin{proof}[Proof of Theorem \ref{thm: UH5.6}] 
By \cite[Theorem 7.6]{ArakiWoods68}, it is sufficient to show that $N$ is an ITPFI-factor. 
Let $\varphi$ be a $\mathbb{Q}$-stable normal faithful state on $N$, let $u_1,\dots, u_n\in \mathcal{U}(N)$, and let $\varepsilon>0$. Choose now $F$, $v_1,\dots, v_n\in \mathcal{U}(F)$ and $w\in \mathcal{U}(N)$ as in Lemma \ref{lem: UH5.5}. Put 
\begin{equation*}
w_k=w^*v_kw,\ \ \ \ \ k=1,\dots, n. 
\end{equation*}
Then $F_1:=w^*Fw$ is a finite-dimensional subfactor of $N$, $w_1,\dots, w_n\in \mathcal{U}(F_1)$ and 
\begin{equation*}
\|w_k-u_k\|_{\varphi}<\varepsilon,\ \ \ \ \ k=1,\dots, n. 
\end{equation*}
Since $F$ is $\sigma^{\varphi}$-invariant, $\varphi=\varphi|_F\otimes \varphi|_{F^{\text{c}}}$ holds. Hence if we put $\varphi_1=w^*\varphi w$, then 
\begin{equation*}
\varphi_1=\varphi_1|_{F_1}\otimes \varphi_1|_{F_1^{\text{c}}}.
\end{equation*}
Since the representing vector of $\varphi_1$ in $\mathcal{P}_N^{\natural}$ is $w^*\xi_{\varphi}w$, we have
\eqa{
\|\varphi-\varphi_1\|&\le \|\xi_{\varphi}-w^*\xi_{\varphi}w\| \|\xi_{\varphi}+w^*\xi_{\varphi}w\|\\
&\le 2\|w\xi_{\varphi}-\xi_{\varphi}w\|\\
&<2\varepsilon.
}
This shows that $\varphi$ satisfies the product condition in Proposition \ref{prop: CW}, and thus $N$ is an ITPFI factor. 
\end{proof}

\end{document}